\documentclass[journal]{IEEEtran}
\usepackage{graphicx}
\graphicspath{{}}
\usepackage[tbtags]{amsmath}
\usepackage{amsfonts}
\usepackage{mathtools}
\mathtoolsset{showonlyrefs=fasle,showmanualtags}
\usepackage{amsthm} 
\theoremstyle{plain}
\newtheorem{theorem}{Theorem}[section]
\newtheorem{lem}[theorem]{Lemma}

\newtheorem{cor}[theorem]{Corollary}
\theoremstyle{definition}
\newtheorem{definition}{Definition}

\theoremstyle{remark}
\newtheorem*{remark}{Remark}

\usepackage{amssymb}
\usepackage{algorithm}
\usepackage{algorithmic}
\usepackage{float}

\usepackage{color,booktabs,colortbl}
\usepackage{latexsym,bm}
\usepackage{comment}
\usepackage{caption}
\usepackage{subcaption}
\newcommand{\norm}[1]{\left\lVert#1\right\rVert}

\DeclareMathOperator*{\argmin}{arg\,min}

\DeclareMathOperator*{\re}{Re}
\DeclareMathOperator*{\im}{Im}
\DeclareMathOperator*{\diag}{diag}
\DeclareMathOperator*{\trace}{Tr}

\DeclareMathOperator*{\dist}{dist}

\renewcommand{\min}{\operatorname*{minimize}}
\begin{document}
\title{On Gradient Descent Algorithm for Generalized Phase Retrieval Problem}
\author{Ji Li, Tie Zhou
\thanks{Ji Li is with LMAM, School of Mathematical Sciences, Peking University, Beijing
100871, China.}
\thanks{Tie Zhou is with LMAM, School of Mathematical Sciences, Peking University, Beijing
100871, China.}
\thanks{Manuscript received ; revised }}

\maketitle

\begin{abstract}
In this paper, we study the generalized phase retrieval problem: to
recover a signal $\bm{x}\in\mathbb{C}^n$ from the measurements
$y_r=\lvert \langle\bm{a}_r,\bm{x}\rangle\rvert^2$,
$r=1,2,\ldots,m$. The problem can be reformulated as a least-squares
minimization problem.  Although the cost function is nonconvex, the
global convergence of gradient descent algorithm from a random
initialization is studied, when $m$ is large enough. We improve the
known result of the local convergence from a spectral
initialization. When the signal $\bm{x}$ is real-valued, we prove that
the cost function is local convex near the solution
$\{\pm\bm{x}\}$. To accelerate the gradient descent, we review and
apply several efficient line search methods. We also perform a
comparative numerical study of the line search methods and the
alternative projection method. Numerical simulations demonstrate the
superior ability of LBFGS algorithm than other algorithms.
\end{abstract}
\section{Introduction}
\label{sec:some-lemas}

Phase retrieval is to recover a complex signal from its Fourier intensity. This
problem arises in many engineering and science applications, such as X-ray crystallography~\cite{Millane1990},
electron microscopy~\cite{Misell1973}, X-ray diffraction
imaging~\cite{Shechtman2015}, optics~\cite{Kuznetsova1988} and
astronomy~\cite{Fienup1982}, just name a few. In these applications,
one often has recorded the Fourier transform intensity of a complex
signal, while the phase information is infeasible. Due to the absence
of Fourier phase, the available information is incomplete. It has been
proved that the one-dimensional phase retrieval problem suffers from
essential nonuniqueness, and the multi-dimensional case is usually
less prone to multiple solutions~\cite{Hayes1982,Hayes1982a}. However,
those theories did not lead efficient recovery algorithm. 

The most widely-used algorithms are based on the method of alternating projections, that are the error
reduction (ER) and its variants, such as HIO~\cite{Fienup1982},
HPR~\cite{Luke2003} and RAAR~\cite{Luke2004}. These iterative
projection methods have combined the oversampling
method~\cite{Miao2000} and additional constraints to increase the probability of
finding a solution. The forementioned algorithms often work well for real-valued signal in
practice and show unsatisfied performance for complex-valued signal~\cite{Fienup1987}. These algorithms have limited recovery abilities due to the issue of convergence to local minimizers. They are identified as the
counterparts of iterative projection methods for convex-set feasible
problem~\cite{Bauschke2002}. Since the intensity constraints in Fourier space is not
convex, so the algorithms do not have theoretical guarantees.

Recently there has been a renewed interest in phase retrieval due to technological advances in measurement systems and theoretical developments in structured signal recovery, see literature~\cite{Candes2013} and references therein. In particular, it is now possible to obtain specific kinds of additional intensity information about the signal, depending on the application. The premise of the multiple measurements approach is that, by carefully redesigning the measurement process, one can potentially resolve the phase ambiguity for phase retrieval. Another advantage of this approach is that the analysis and developed algorithms are independent of the dimensional of the signal, as opposed to the alternative projection methods. Mathematically, we consider the generalized phase retrieval problem. It is to 
find a vector $\bm{z}\in\mathbb{C}^n$, given that
\begin{equation}
\label{eq:1}
  y_r=\lvert \langle \bm{a}_r,\bm{z}\rangle\rvert^2,\quad r=1,2,\ldots,m,
\end{equation}
where $\bm{a}_r\in\mathbb{C}^n$ are known sampling vectors, and
$y_r\in \mathbb{R}$ are the intensity measurements. For the recent progress on the generalized phase problem,
we refer the reader to the survey
papers~\cite{Shechtman2015,Jaganathan2015}.

The problem~\eqref{eq:1} can be reformulated as
an NP-hard matrix rank minimization problem by lifting a vector to a
rank-one matrix. Semi-definite programs (SDP), such as
PhaseLift~\cite{Candes2013,Candes2012} and
PhaseCut~\cite{Waldspurger2015,Fogel2013}, are used to solve its
convex relaxation problem based on two different formulations. PhaseLift is to solve the
following convex
optimization problem:
\begin{align*}
  \min\quad &\trace(X)\\
\text{s.t.}\quad&\mathcal{A}_r(X) = y_r,\quad r=1,2,\ldots,m,
\end{align*}
where
$\mathcal{A}_r(X)=\trace(\bm{a}_r\bm{a}_r^*X)=\trace(\bm{a}_r\bm{a}_r^*\bm{x}\bm{x}^*)=\lvert\langle
\bm{a}_r,\bm{x}\rangle\rvert^2=y_r$. The main advantage is that the convex optimization has
theoretical guarantees and efficient numerical methods. A nature
question is how/when the solution of PhaseLift/PhaseCut is also
exactly a solution of the original phase problem and how to design the
sampling vector to guarantee the recovery. It is shown that the
required sampling complexity $m$ is $\mathcal{O}(n)$ for Gaussian
model~\cite{Candes2012} and $\mathcal{O}(n\log^2 n)$ for coded
diffraction pattern model~\cite{Candes2014,Gross2015}. While in principle
SDP-based methods offer tractable solution, they become
computationally prohibitive as the dimension of the signal
increases. So recently, authors of~\cite{Candes2015} reformulated \eqref{eq:1} as a least-squares problem. A solution to the problem
\eqref{eq:1} is any solution to the
optimization problem
\begin{equation}
  \label{eq:2}
  \min_{\bm{z}\in\mathbb{C}^n}\quad f(\bm{z}):=\frac{1}{2m}\sum_{r=1}^{m}(\lvert \bm{a}_r^*\bm{z}\rvert^2-y_r)^2.
\end{equation}

The fixed-stepsize gradient descent algorithm (called Wirtinger Flow
in their paper) is applied to solve \eqref{eq:2} in literature~\cite{Candes2015}. The local
convergence to a global minimizer is also shown if the initialization
is near the global minimizer.

In this paper, we prove the global convergence of gradient
descent algorithm with an appropriate stepsize from a random initialization instead of a good
spectral initialization proposed in literature~\cite{Candes2015}. We
find that the algorithm converges to a global minimizer from a random
initialization when $m$ is large enough (with the complexity $n\log n$ for Gaussian model and $n\log^3n$ for CDP
model) and that all local minimizers of \eqref{eq:2} are global minimizers with high probability. In addition, if the signal is real-valued, the local convexity
of the least-squares cost function in \eqref{eq:2} is proved. For numerical algorithm,
to accelerate the convergence rate, more efficient line search methods
for minimizing function of complex variables, such as nonlinear
conjugate gradient (NCG) and Limited-memory BFGS (LBFGS), are
considered. At the same time of preparing
this paper, we find that literature~\cite{Sun2016} present a
geometrical analysis of phase retrieval and apply the trust-region
method (TRM) to solve the optimization problem. As opposed to our
first-order algorithms, solving a linear equation is needed at each
iteration of solving the subproblem of TRM, so our algorithms is
better in terms of computational cost.

The remainder of this paper is organized as follows. In
Section~\ref{sec:algor-via-minim}, we recall the common least-squares
cost function and derive the gradient and Hessian expressions by
Fr\'{e}chet derivative. The expressions are identical to that from the
definition of the Wirtinger derivative. It follows gradient descent
algorithms with an appropriate stepsize. Besides, some accelerating
line search methods are reviewed in
Section~\ref{sec:unconstr-optim-meth}. In
Section~\ref{sec:two-cases:-main} we present the two main results: one
states the convergence to a solution to~\eqref{eq:1} if provided
enough measurements, the other one states the local convexity of the
cost function~\eqref{eq:2} provided the signal is real-valued.  In
Section~\ref{sec:numer-simul} we test some synthetic models to study the empirical sampling
complexity and the empirical performance of the line search methods
for solving \eqref{eq:1}. Section~\ref{sec:conclusion} concludes the
paper with some discussions.
\section{Gradient Descent and Other First-order Algorithms}
\label{sec:algor-via-minim}
We focus on the common least squares cost objective function to measure the misfit between the
observed data and predicted data. Recall equation
\eqref{eq:1},
if we denote matrix
$A=[\bm{a}_1,\bm{a}_2,\ldots,\bm{a}_m]^*\in\mathbb{C}^{m\times n}$,
objective $f(\bm{z})$ can be rewritten as
\begin{equation}
  \label{eq:3}
  f(\bm{z})=\frac{1}{2m}\norm{\lvert A\bm{z}\rvert^2-\bm{y}}^2,
\end{equation}
where $\lvert\cdot\rvert$ is the componentwise absolute value, and
$\bm{y}=[y_1,y_2,\ldots,y_m]^T$.

\subsection{Gradient and Hessian}
\label{sec:gradient-hessian}

We derive the gradient and Hessian's analytical expressions of the function
$f(\bm{z})$ defined by \eqref{eq:2} in this subsection. Note that the decision variable
$\bm{z}$ is complex and $f(\bm{z})$ is real-valued, so the terminologies,
gradient and Hessian, can be viewed from the perspective of Wirtinger
derivatives~\cite[see Section 6]{Candes2015} or
$\mathbb{C}-\mathbb{R}$ calculus~\cite{Kreutz-Delgado2009}. Here we
use the Fr\'{e}chet derivative for an operator defined in Hilbert
space to deduce its gradient and Hessian, instead of  using the
Wirtinger derivative. We find that the expressions resulting from the
two approaches are identical and the Fr\'{e}chet derivative approach
is more compact and convenient than Wirtinger derivative approach. 

Function $f(\bm{z})$ (see~\eqref{eq:3}) can be recognized as an operator
defined from Hilbert space $\mathbb{C}^n$
to $\mathbb{R}$. Its Fr\'{e}chet derivative at point $\bm{z}$ is 
\begin{equation}
  \label{eq:4}
  Df[\bm{z}](\bm{h}) = \frac{1}{m}\left\langle Dg[\bm{z}](\bm{h}), \lvert A\bm{z}\rvert^2-\bm{y}\right\rangle,
\end{equation}
where $g(\bm{z})=\lvert A\bm{z}\rvert^2-\bm{y}$. Substituting
\begin{equation}
  \label{eq:5}
  Dg[\bm{z}](\bm{h}) = 2\re(A\bm{z}\circ \overline{A\bm{h}})
\end{equation}
into~\eqref{eq:4}, we have that
\begin{align}
  Df[\bm{z}](\bm{h}) &=2m^{-1}\re\left\langle A\bm{h}\circ
                       \overline{A\bm{z}}, \lvert
                       A\bm{z}\rvert^2-\bm{y}\right\rangle  \label{eq:6}\\
& = 2m^{-1}\re\left\langle A\bm{h}, A\bm{z}\circ (\lvert
  A\bm{z}\rvert^2-\bm{y})\right\rangle\nonumber\\
&=2m^{-1}\re\left\langle\bm{h},A^*\bigl(A\bm{z}\circ (\lvert
  A\bm{z}\rvert^2-\bm{y})\bigr)\right\rangle.\nonumber
\end{align}

In a similar way, the Hessian operator (Hessian-vector
multiplication) can be derived. Differentiating~\eqref{eq:6}, we get 
\begin{align*}
  D^2f[\bm{z}](\bm{h},\bm{q}) &=2m^{-1}\re\left\langle
                                \overline{A\bm{q}}\circ A\bm{h},\lvert
                       A\bm{z}\rvert^2-\bm{y}\right\rangle \\
&{}+
                                2m^{-1}\re\left\langle
                                \overline{A\bm{z}}\circ
                                A\bm{h},2\re(A\bm{z}\circ
                                \overline{A\bm{q}})\right\rangle\\
&\hspace*{-1cm}=2m^{-1}\re\left\langle A\bm{h}, (2\lvert A\bm{z}\rvert^2-\bm{y})\circ
  A\bm{q}+(A\bm{z})^2\circ \overline{A\bm{q}}\right\rangle.
\end{align*}
Let $\bm{q}=\bm{h}$, it yields
\begin{multline}
  \label{eq:7}
D^2f[\bm{z}](\bm{h},\bm{h}) \\=2m^{-1}\re\left\langle
\bm{h},A^*\bigl((2\lvert A\bm{z}\rvert^2-\bm{y})\circ
A\bm{h}\bigr)+A^*\bigl((A\bm{z})^2\circ
\overline{A\bm{h}}\bigr)\right\rangle\\
:=2\re\left\langle \bm{h},\mathcal{H}_f[\bm{z}](\bm{h})\right\rangle,
\end{multline}
where $\mathcal{H}_f[\bm{z}](\cdot)$ is the Hessian operator.

According to the Taylor expansion
\begin{align}
f(\bm{z}+\bm{h})&=f(\bm{z})+Df[\bm{z}](\bm{h})+\frac{1}{2}D^2f[\bm{z}](\bm{h},\bm{h})+\text{h.o.t},\notag\\
&\hspace*{-1cm} := f(\bm{z}) + 2\re\left\langle \bm{h},\nabla f(\bm{z})\right\rangle + \re\left\langle \bm{h}, \mathcal{H}_{f}[\bm{z}](\bm{h})\right\rangle
+ \text{h.o.t},\label{eq:8}
\end{align}
so the gradient and Hessian operator are 
\begin{subequations}
\begin{align}
  \label{eq:9}
  \nabla f(\bm{z}) &= m^{-1}A^*\bigl(A\bm{z}\circ (\lvert
  A\bm{z}\rvert^2-\bm{y})\bigr)\\
\mathcal{H}_f(\bm{h}) & = m^{-1}A^*\bigl((2\lvert A\bm{z}\rvert^2-\bm{y})\circ
A\bm{h}\bigr)+A^*\bigl((A\bm{z})^2\circ
\overline{A\bm{h}}\bigr).
\end{align}
\end{subequations}
From the definition of Hessian, the Hessian matrix is given by
\begin{multline*}
  \nabla^2 f(\bm{z}) =\\
  \begin{bmatrix}
    m^{-1}A^*\diag\bigl(2\lvert A\bm{z}\rvert^2-\bm{y}\bigr)A&
    m^{-1}A^*\diag\bigl((A\bm{z})^2\bigr)\overline{A}\\
    m^{-1}A^T\diag\bigl((\overline{A\bm{z}})^2\bigr)A& m^{-1}A^T\diag\bigl(2\lvert A\bm{z}\rvert^2-\bm{y}\bigr)\overline{A}
  \end{bmatrix}.
\end{multline*}

For easy reference, the gradient and Hessian can be expressed in components
$\bm{a}_r$ instead of the above compact form. They have the following forms:
\begin{subequations}
  \begin{align}
    \label{eq:10}
    \nabla f(\bm{z}) &=\frac{1}{m}\sum_{r=1}^m \left(\lvert
                       \bm{a}_r^*\bm{z}\rvert^2-y_r\right)(\bm{a}_r^*\bm{z})\bm{a}_r,\\
\nabla^2 f(\bm{z}) &\\
&\hspace*{-1cm}=\frac{1}{m}\sum_{r=1}^m
    \begin{bmatrix}
      \left(2\lvert
        \bm{a}_r^*\bm{z}\rvert^2-y_r\right)\bm{a}_r\bm{a}_r^*&
      \left(\bm{a}_r^*\bm{z}\right)^2\bm{a}_r\bm{a}_r^T\\
\left(\overline{\bm{a}_r^*\bm{z}}\right)^2\overline{\bm{a}_r}\bm{a}_r^*&\left(2\lvert
        \bm{a}_r^*\bm{z}\rvert^2-y_r\right)\overline{\bm{a}_r}\bm{a}_r^T
    \end{bmatrix}.
  \end{align}
\end{subequations}

\subsection{Algorithms: Gradient Descent and Accelerating Strategies}
\label{sec:unconstr-optim-meth}

Sine we have deduced the expression of gradient, we can construct
iterative algorithms which are only based on the gradient
information. Common optimization algorithms are
constructed for real-valued function with real variables. Since the cost
function~\eqref{eq:2} is real-valued with complex
variables, the optimization is usually carried out with respect to the
real and imaginary part of these variables. Here we consider the straightforward extension of optimization of
function of complex variables. 

Note that the
cost function~\eqref{eq:2} is nonconvex, so all the line search
methods are generally guaranteed to converge to the local minimizers. It is known that
Fourier phase retrieval is prone to local minimizer, which is far away
from the global minimizer. Using multiple measurements by random
masks, we have the advantage that all local minimizers are global
minimizers with high probability, see
Section~\ref{sec:two-cases:-main}. So the line search methods with
local convergence, such as nonlinear conjugate gradient (NCG) and
limited-memory BFGS (LBFGS), can be applied to the generalized phase
retrieval~\eqref{eq:1} with efficient performance and ensure global
convergence in practice.

Line search
methods construct a sequence
\begin{equation}
  \label{eq:12}
  \bm{z}_{k+1} = \bm{z}_k+\alpha_k\bm{d}_k.
\end{equation}
The basic idea is first to choose a descent direction
$\bm{d}_k\in\mathbb{C}^n$, then to refine the iteration with some line
search scheme to choose the appropriate step length
$\alpha_k\in\mathbb{R}$ at $k$th
iteration. The most simple iteration (a.k.a
gradient descent) is stated as following: start with an initialization $\bm{z}_0\neq\bm{0}$, and inductively update
\begin{equation*}
  \bm{z}_{k+1}=\bm{z}_k-\alpha_k\bm{g}_k,
\end{equation*}
where $\alpha_k$ is the stepsize and $\bm{g}_k=\nabla f(\bm{z}_k)$, i.e., taking descent direction $\bm{d}_k=-\bm{g}_k$. 

To accelerate the rate of convergence, nonlinear
conjugate gradient (NCG) method is widely used. The conjugate gradient
direction $\bm{d}_k$ is generated by the recurrence relation
\begin{equation}
  \label{eq:16}
  \bm{d}_k = -\bm{g}_k + \beta_k \bm{d}_{k-1},
\end{equation}
where $\bm{d}_0 = \bm{0}$. There are a variety of options to choose
parameter $\beta_k$ for nonlinear problem~\cite{Steihaug1983}. In this
paper, we take the Hestenes-Stiefel form
\begin{equation}
  \label{eq:17}
  \beta_k^{HS} = -\frac{\re\left(\bm{g}_k^*(\bm{g}_k-\bm{g}_{k-1})\right)}{\re\left(\bm{d}_{k-1}^*(\bm{g}_k-\bm{g}_{k-1})\right)}.
\end{equation}

Although Newton algorithm has two-order convergence rate near the
minimizer, it is not suit for large scale problem, since solving the
Newton equation at each iteration is required. We consider the famous
LBFGS method, which 
is appropriate for large-scale problem, and  the descent direction $\bm{d}_k$ can be obtained by the easy
two-loop recursion, which is described in
Algorithm~\ref{alg:1}.
\begin{algorithm}[H]
  \caption{LBFGS two-loop recursion\label{alg:1}}
  \begin{algorithmic}
    \REQUIRE $\bm{g}_k$, $\bm{s}_i=\bm{z}_{i+1}-\bm{z}_i$,
    $\bm{y}_i=\bm{g}_{i+1}-\bm{g}_i$, $\rho_i =
    \frac{1}{\re(\bm{y}_i^*\bm{s}_i)}$, for $i=k-s,\ldots,k-1$, $s$ is the number of storing pair.
    \ENSURE $\bm{d}$
\STATE $\bm{d}\leftarrow -\bm{g}_k$
\FOR {$i=k-1,k-2,\ldots,k-s$}
\STATE $\alpha_i = \rho_i\re(\bm{s}_i^*\bm{d})$
\STATE $\bm{d}\leftarrow \bm{d}-\alpha_i \bm{y}_i$
\ENDFOR
\STATE $\bm{d}\leftarrow \gamma\bm{d}$, with the scaling suggested
by Shanno and Phua $\gamma=\frac{\re(\bm{y}_{k-1}^*
  \bm{s}_{k-1})}{\bm{y}_{k-1}^*\bm{y}_{k-1}}$
\FOR {$i=k-s,k-s+1,\ldots,k-1$}
\STATE $\beta\leftarrow \rho_i\re(\bm{y}_i^*\bm{d})$
\STATE $\bm{d}\leftarrow \bm{d}+(\alpha_i-\beta)\bm{s}_i$
\ENDFOR
  \end{algorithmic}
\end{algorithm}

With holding global convergence to
local minimizer, step length $\alpha_k$ is not arbitrary. 
For the steepest gradient descent, we strictly characterize the choice
strategy of stepsize $\alpha_k$, which will lead to global
convergence. In next section, we see that our convergence analysis for the gradient descent is based on two
idealizations: (i) the solution $\bm{x}$ is known a priori; and (ii)
the stepsize $\alpha_k$ is obtained by solving a equation with degree of
three, which render the numerical
algorithm impractical. For optimization algorithms, there is
usually a chosen strategy of $\alpha_k$, which are known as Wolfe conditions:
\begin{subequations}
  \begin{equation}
    \label{eq:14}
    f(\bm{z}_k + \alpha_k \bm{d}_k)\leq f(\bm{z}_k) + c_1\alpha_k \re(\bm{d}_k^*\bm{g}_k)
  \end{equation}
and
\begin{equation}
  \label{eq:15}
  \re (\bm{d}_k^*\bm{g}_{k+1}) \geq c_2 \re(\bm{d}_k^*\bm{g}_k),
\end{equation}
\end{subequations}
where condition $0<c_1<c_2<1$ is satisfied. Generally we take $c_1 =
10^{-4}$, $c_2 = 0.9$ as commended in
book~\cite{Nocedal2006}. Equations~\eqref{eq:14} and~\eqref{eq:15} are
known as the sufficient decrease and curvature condition
respectively. We call gradient descent with stepsize $\alpha_k$ by
Wolfe conditions steepest gradient (SD) algorithm.

Our theoretical results about global convergence is only applied to the
simple gradient descent algorithm. But as we will see in
Section~\ref{sec:numer-simul}, the numerical performance of NCG and
LBFGS is superior than the steepest descent (a.k.a, gradient descent with line search scheme by
Wolfe conditions), even the gradient descent we consider here. The
complexities of line search scheme and the choice of descent
direction of NCG and LBFGS hinder the analysis of the convergence.
\section{Theoretical Results}
\label{sec:two-cases:-main}

We assume that the sampling vectors of the model setup are in the Gaussian~\cite{Candes2012} or coded diffraction pattern (CDP)~\cite{Candes2014,Candes2015} models, which are defined
below. Gaussian model has more theoretical interests than
CDP, but the latter is a more physical realizable model. 

If $\bm{a}_r\in \mathbb{C}^n$ are drawn from
$N(\bm{0},\bm{I}/2)+iN(\bm{0},\bm{I}/2)$, we say that sampling vectors
follow the Gaussian model. In CDP model, we collect multiple diffraction patterns with
different masks $\bm{d}_l\in\mathbb{C}^n$. The
observation data are 
\begin{multline*}
  \label{eq:11}
  y_{l,k} = \lvert\langle \bm{f}_k,
  \bm{d}_l\circ\bm{x}\rangle\rvert^2=\lvert\bm{f}_k^*D_l\bm{x} \rvert^2, \\
  \quad l=1,2,\ldots,L,\quad k = 1,2,\ldots,n,\quad D_l=\diag(\bm{b}_l),
\end{multline*}
where $\bm{f}_k^*$ is the $k$th row of the discrete Fourier transform matrix, i.e.,
$\bm{f}_k=[1,\exp(i2\pi(k-1)/n),\ldots,\exp(i2\pi(n-1)(k-1)/n)]^T$. And each entry of mask $\bm{b}_l$ samples
from a distribution $b$. 
Assume the entry $b=b_1b_2$, where $b_1$ and $b_2$ are independent and
distributed as: $b_1$ is sampled from $\{-1,1,-i,i\}$ with equal
probability $1/4$, and $b_2$ from $\{\sqrt{2}/2,\sqrt{3}\}$ with
probability $4/5$ and $1/5$ respectively. This pattern is called
octanary in paper~\cite{Candes2015}.

We present the main results in this section and put the proofs and technical lemmas in the Appendix of this paper.

\subsection{Global Convergence of Gradient Descent}
\label{sec:glob-conv-grad}

\begin{definition}
  Let $\bm{x}\in\mathbb{C}^n$ be any solution to the phase retrieval
  problem~\eqref{eq:1}. For each $\bm{z}\in\mathbb{C}^n$, define
  \begin{equation*}
\phi(\bm{z})=\argmin_{\phi\in[0,2\pi]}\norm{\bm{z}-e^{i\phi}\bm{x}},\quad
\bm{h(\bm{z})}=\frac{\bm{z}-e^{i\phi(\bm{z})}\bm{x}}{\norm{\bm{z}-e^{i\phi(\bm{z})}\bm{x}}},
  \end{equation*}
then the distance of $\bm{z}$ to the solution set $\chi=\{\bm{x}e^{i\theta}\colon
\theta\in[0,2\pi)\}$ is 
\begin{equation*}
  \dist(\bm{z},\bm{x})\colon=\dist(\bm{z},\chi)=\norm{\bm{z}-\bm{x}e^{i\phi(\bm{z})}}.
\end{equation*}
For real case, 
\begin{equation*}
   \dist(\bm{z},\bm{x})\colon=\dist(\bm{z},\chi)=\operatorname{min}\{\norm{\bm{z}-\bm{x}},\norm{\bm{z}+\bm{x}}\}.
\end{equation*}
\end{definition}

It is not difficult to see that
$\im(\bm{h}^*\bm{x}e^{i\phi(\bm{z})})=0$ and
$\re(\bm{h}^*\bm{x}e^{i\phi(\bm{z})})=\norm{\bm{x}}$.

\begin{lem}[\cite{Candes2015}, Lemma 7.1]
\label{lem:1}
  Assume that the solution $\bm{x}$ of the phase retrieval problem is
  independent from the sampling vectors. Furthermore, the sampling
  vectors $\bm{a}_r$ are distributed according to either the Gaussian
  or admissible CDP model. Then
  \begin{multline*}
    \label{eq:22}
    \mathbb{E}[\nabla^2 f(\bm{x})] =
                            \begin{bmatrix}
                              \norm{\bm{x}}^2I+\bm{x}\bm{x}^*&2\bm{x}\bm{x}^T\\
                              2\bar{\bm{x}}\bm{x}^*&\norm{\bm{x}}^2I+\bar{\bm{x}}\bm{x}^T
                            \end{bmatrix}
\\
    = \norm{x}^2I+\frac{3}{2}
    \begin{bmatrix}
      \bm{x}\\\overline{\bm{x}}
    \end{bmatrix}
    \begin{bmatrix}
      \bm{x}^*&\bm{x}^T
    \end{bmatrix}-\frac{1}{2}
    \begin{bmatrix}
      \bm{x}\\-\overline{\bm{x}}
    \end{bmatrix}
    \begin{bmatrix}
      \bm{x}^*&-\bm{x}^T
    \end{bmatrix}.
  \end{multline*}
\end{lem}
\begin{lem}
\label{lem:3}
  For the expectation of the Hessian $\mathbb{E}[\nabla^2
  f(\bm{x})]$, then its eigenvalues are
  $\{4\norm{\bm{x}}^2,\underbrace{\norm{\bm{x}}^2,\ldots,\norm{\bm{x}}^2}_{2n-2~
  \text{terms}},0\}$. Furthermore the expectation matrix is semi-definite. In
the real case, i.e., the solution
  $\bm{x}\in\mathbb{R}^n$, the expectation of the Hessian $\mathbb{E}[\nabla^2 f(\bm{x})]$ is $\norm{\bm{x}}^2I+3\bm{x}\bm{x}^T$. It is definite, since 
its eigenvalues are
$\{4\norm{\bm{x}}^2,\underbrace{\norm{\bm{x}}^2,\ldots,\norm{\bm{x}}^2}_{n-1~\text{terms}}\}$.
\end{lem}
\begin{remark}
  Obviously, vector $
  \begin{bmatrix}
    \bm{x}\\\bar{\bm{x}}
  \end{bmatrix}
$ and $ \begin{bmatrix}
    i\bm{x}\\-i\bar{\bm{x}}
  \end{bmatrix}$ are the eigenvectors corresponding to eignevalues
  $4\norm{\bm{x}}^2$ and $0$ respectively by
calculus.
\end{remark}
\begin{lem}[\cite{Candes2015}, Lemma 7.4]
\label{lem:4}
 Assume the vectors $\bm{a}_r$ are
  distributed according to either the Gaussian or admissible CDP model
  with a sufficiently large number of measurements. This means that the
  number of samples obeys $m\geq c(\delta)n\log n$ in the Gaussian
  model and the number of patterns obeys $L\geq c(\delta)\log^3 n$ in
  the CDP model. Then
  \begin{equation*}
    \norm{\nabla^2 f(\bm{x})-\mathbb{E}[\nabla^2 f(\bm{x})]}\leq \delta\norm{\bm{x}}^2
  \end{equation*}
holds with probability at least $1-10e^{-\gamma n}-8/n^2$ and
$1-(2L+1)/n^3$ for the Gaussian and CDP models respectively.
\end{lem}
\begin{lem}
\label{lem:8}
  Let $\bm{z}\in\mathbb{C}^n$ be a fixed vector
  independent of the sampling vectors. We have
  \begin{equation*}
    \mathbb{E}[\nabla f(\bm{z})]=(2\norm{\bm{z}}^2-\norm{\bm{x}}^2)\bm{z}-(\bm{x^*}\bm{z})\bm{x}.
  \end{equation*}
\end{lem}
\begin{lem}
\label{lem:9}
  In the setup of Lemma~\ref{lem:4}, let $\bm{z}\in\mathbb{C}^n$ be a fixed vector
  independent of the sampling vectors. Then
  \begin{equation}
    \label{eq:23}
    \norm{\nabla f(\bm{z})-\mathbb{E}[\nabla f(\bm{z})]}\leq \frac{\delta}{2}\norm{\mathbb{E}[\nabla
      f(\bm{z})]}
  \end{equation}
holds with probability at least $1-20e^{-\gamma m}-4m/n^4$ in the
Gaussian model and $1-(4L+2)/n^3$ in the CDP model. Furthermore, if
$\bm{z}$ obeying
  $\dist(\bm{z},\bm{x})\leq \norm{\bm{x}}/2$, then 
\begin{equation*}
    \norm{\nabla f(\bm{z})-\mathbb{E}[\nabla f(\bm{z})]}\leq 4\delta\dist(\bm{z},\bm{x})\norm{\bm{x}}^2
  \end{equation*}
holds with the same high probability.
\end{lem}
\begin{lem}
\label{lem:10}
 In the setup of Lemma~\ref{lem:4}, then $\mathbb{E}[\nabla
 f(\bm{z})]$ and $\bm{h}$ share the same direction and
 $\mathbb{E}[\nabla f(\bm{z})]=\bm{0}$ if and only if $\bm{z}$ is a
 solution of the phase retrieval problem \eqref{eq:1}, which holds
 with high probability. Furthermore, the angle between $\nabla f(\bm{z})$ and
  $\bm{z}-\bm{x}e^{i\phi(\bm{z})}$ is below $\arcsin(\delta/2)$, i.e., the
  following 
  \begin{equation*}
    \re\left\langle \nabla
      f(\bm{z}),\bm{z}-\bm{x}e^{i\phi(\bm{z})}\right\rangle\geq \sqrt{1-\frac{\delta^2}{4}}
    \norm{\nabla f(\bm{z})}\norm{\bm{z}-\bm{x}e^{i\phi(\bm{z})}}
  \end{equation*}
holds with high probability.
\end{lem}
\begin{remark}
\label{remark:2}
  From Lemma~\ref{lem:10}, $\nabla f(\bm{z})=\bm{0}$ if and only if
  $\bm{z}=\bm{0}$ or $\bm{z}$ is a
 solution of the phase retrieval problem \eqref{eq:1} with high
 probability. So all local minimizers are global minimizers with high probability, which
 facilitates the line search methods. 
\end{remark}
\begin{theorem}[Global Convergence]
\label{thm:main}
  Let $\bm{x}$ be a solution to the generalized phase retrieval
  problem~\eqref{eq:1} and the number of samples $m$ obeys $m\geq
  c_1(\delta)n\log n$ in Gaussian model or the number of patterns obeys $L\geq
  c_2(\delta)\log^3n$ in the CDP model, where $c_1,c_2$ are sufficiently large
  numerical constants. For the following
  gradient descent updating scheme ($\bm{z}_0\neq 0$)
  \begin{equation*}
    \bm{z}_{k+1}=\bm{z}_k-\alpha_k\nabla f(\bm{z}_k),
  \end{equation*}
the (strict) descent property
\begin{equation*}
  \dist(\bm{z}_{k+1},\bm{x})< \dist(\bm{z}_{k},\bm{x}) 
\end{equation*}
holds with probability at least $1-20e^{-\gamma m}-4m/n^4$ in Gaussian model and
  $1-(4L+2)/n^3$ in CDP model, if the stepsize $\alpha_k$
  satisfies
\begin{multline*}
 0<\alpha_k\\<\operatorname{min}\left\{2\sqrt{1-\frac{\delta^2}{4}}\frac{\norm{\nabla f(\bm{z}_k)}}{6(2+\delta)\norm{\bm{x}}^2},2\sqrt{1-\frac{\delta^2}{4}}\sqrt[3]{\frac{\norm{\nabla f(\bm{z}_k)}}{6(2+\delta)}}\right\}.
\end{multline*}
Furthermore, we denote the intersection of the polynomial of
degree three $P_1(t)=(2+\delta)t(t^2+3t\norm{\bm{x}}+2\norm{\bm{x}}^2)$ and
$y=\norm{\nabla f(\bm{z}_k)}$ as $t_1$ and take stepsize
$\alpha_k=\sqrt{1-\delta^2/4}t_1$, then the geometrical convergence is
ensured, that is
\begin{equation*}
  \dist(\bm{z}_{k+1},\bm{x})\leq\left(\frac{2\sqrt{2}\delta}{2-\delta}+\frac{\delta}{\sqrt{2}}\right)\dist(\bm{z}_{k},\bm{x}).
\end{equation*}
\end{theorem}

It is inconvenient to determine $\alpha_k$ by finding a root of
function, we can relax the $\alpha_k$ in two cases. 
\begin{cor}
\label{cor:cor}
  In the setup of Theorem~\ref{thm:main}, if
  $\dist(\bm{z_k},\bm{x})\leq \norm{\bm{x}}/5$, then
  $\alpha_k$ can be taken as $\sqrt{1-\delta^2/4}\frac{25\norm{\nabla
      f(\bm{z}_k)}}{66(2+\delta)\norm{\bm{x}}^2}$. we have 
  \begin{equation*}
    \dist(\bm{z}_{k+1},\bm{x})\leq\left(\frac{\sqrt{2}(58\delta+16)}{25(2-\delta)}+\frac{\delta}{\sqrt{2}}\right)\dist(\bm{z}_{k},\bm{x});
  \end{equation*}
If $\dist(\bm{z_k},\bm{x})\geq 2\norm{\bm{x}}$, and $\alpha_k=\sqrt{1-\delta^2/4}\sqrt[3]{\frac{\norm{\nabla
      f(\bm{z}_k)}}{3(2+\delta)}}$, we have
\begin{equation*}
  \dist(\bm{z}_{k+1},\bm{x})\leq\left(\frac{\sqrt{2}(4\delta+4)}{3(2-\delta)}+\frac{\delta}{\sqrt{2}}\right)\dist(\bm{z}_{k},\bm{x}).
\end{equation*}
\end{cor}

\begin{remark}
  \begin{enumerate}
  \item Theorem~\ref{thm:main} tells us that the gradient descent algorithm with appropriate
stepsize converges to a solution of~\eqref{eq:1} when given enough
many measurements. It is an extension of the local convergence of
gradient descent method, see~\cite[Theorem
3.3]{Candes2015}. Our proof is more geometrical than the proof in literature~\cite{Candes2015}.
\item From the Lemma~\ref{lem:10}, for any nonzero vector $\bm{z}\in\mathbb{C}^n$
  outside the solution set $\chi$, $\nabla f(\bm{z})\neq 0$ hold with
  probability at least $1-20e^{-\gamma m}-4m/n^4$ in Gaussian model and
  $1-(4L+2)/n^3$ in CDP model. Thus, there is no saddle point of $f(\bm{z})$ with high
probability. When $m$
is large enough, $\{\bm{x}e^{i\theta}\}$ are the only local
minimizers, and also global minimizers. This fact makes the
optimization approach easily to find a solution, it is also noticed by
authors of literature~\cite{Sun2015}. 
\item As opposed to Wirtinger Flow (WF) algorithm, we do iterate from a random
initialization without a careful choice of initialization. The algorithm
looses at most a logarithmic factor in the sampling complexity as the
WF\@. It is an open question whether the complexity can be proportional
to $n$.
\item Although the norm of solution $\norm{\bm{x}}$ is not known a
  priori, it has an uniform upper bound with high probability in
  Gaussian model, since we have $\mathbb{E}[\lvert\bm{a}_r^*\bm{x} \rvert^2]=\norm{\bm{x}}^2$.
\end{enumerate}
\end{remark}

It is known that the phase retrieval problem is less difficult
to solve when the target signal $\bm{x}$ is real. We shall give the
theoretical aspect why it happens. It is due to
the local convexity of $f(\bm{z})$ when $\bm{z}$ is in a
neighborhood of the solution $\bm{x}$. 

\subsection{Local Convexity with Real-Valued Signal}
\label{sec:real-case}

We assume the target signal $\bm{x}\in\mathbb{R}^n$ and the sampling
vectors $\bm{a}_r$ are drawn from the Gaussian model. The gradient and Hessian of objective
$f$ have the following form (up to a scale factor $1/2$):
\begin{subequations}
  \begin{align}
      \nabla f(\bm{z}) &=\frac{1}{m}\sum_{r=1}^m \re\Bigl(\left(\lvert
                       \bm{a}_r^*\bm{z}\rvert^2-y_r\right)(\bm{a}_r^*\bm{z})\bm{a}_r\Bigr),\\
\nabla^2 f(\bm{z}) &=\frac{1}{m}\sum_{r=1}^m\re\Bigl(
      \left(2\lvert
        \bm{a}_r^*\bm{z}\rvert^2-y_r\right)\bm{a}_r\bm{a}_r^*+
      \left(\bm{a}_r^*\bm{z}\right)^2\bm{a}_r\bm{a}_r^T\Bigr).
  \end{align}
\end{subequations}
\begin{lem}
\label{lem:5}
  Let $\bm{z}\in\mathbb{R}^n$ be a fixed
  vector independent of the sampling vectors. Then we have
   \begin{multline*}
    \mathbb{E}[\nabla^2 f(\bm{z})]=
    \begin{bmatrix}
      A&2\bm{z}\bm{z}^T\\
      2\bar{\bm{z}}\bm{z}^*&\overline{A}
    \end{bmatrix},\\\text{ with } A=\left(2\norm{\bm{z}}^2-\norm{\bm{x}}^2\right)I+2\bm{z}\bm{z}^*-\bm{x}\bm{x}^*.
  \end{multline*}
For real case, the expectation matrix for a fixed vector $\bm{z}$,
\begin{equation*}
    \mathbb{E}[\nabla^2 f(\bm{z})]=
      \left(2\norm{\bm{z}}^2-\norm{\bm{x}}^2\right)I+4\bm{z}\bm{z}^T-\bm{x}\bm{x}^T.
  \end{equation*}
\end{lem}
\begin{lem}[\cite{Candes2015}, Lemma 7.3]
\label{lem:7}
  Assume $\bm{u},\bm{v}\in\mathbb{C}^n$ are fixed vectors obeying $\norm{\bm{u}}=\norm{\bm{v}}=1$ which are independent of the sampling vectors. Furthermore, assume the measurement vectors $\bm{a}_r$ are distributed according to the Gaussian model. Then
  \begin{align*}
    \mathbb{E}[\lvert\bm{a}_r^*\bm{u}\rvert^{2k}]&=k!\\
    \mathbb{E}[\re(\bm{u}^*\bm{a}_r\bm{a}_r^*\bm{v})\lvert\bm{a}_r^*\bm{u} \rvert^2]&=2\re(\bm{u}^*\bm{v})\\ \mathbb{E}[(\re(\bm{u}^*\bm{a}_r\bm{a}_r^*\bm{v}))^2]&=\frac{1}{2}+\frac{3}{2}(\re(\bm{u}^*\bm{v}))^2-\frac{1}{2}(\im(\bm{u}^*\bm{v}))^2.
  \end{align*}
\end{lem}
\begin{theorem}[Convexity on Expectation (Asymptotic Convexity)]
\label{thm:2}
  Let $\bm{z}\in\mathbb{R}^n$ be a fixed
  vector independent of the sampling vectors, we have that
  $\mathbb{E}[f(\bm{z})]$ is convex in the ellipse
  $\dist(\bm{z},\bm{x})\leq \norm{\bm{x}}/12$.
\end{theorem}

In this proof, we use the following fact that the eigenvalues of matrix
$(\bm{x}^T\bm{w})I+\bm{x}\bm{w}^T+\bm{w}\bm{x}^T$ are
\[\{\underbrace{\bm{x}^T\bm{w},\ldots,\bm{x}^T\bm{w}}_{n-2~\text{terms}},2\bm{x}^T\bm{w}\pm
  \norm{\bm{x}}\}.\]
It can be obtained from the Lemma~\ref{lem:2}.

\begin{theorem}[Strong Convex]
\label{thm:3}
  In the setup of Theorem~\ref{thm:main}, for all
  $\bm{z}\in\mathbb{R}^n$ in the ellipse around $\bm{x}$, more
  specified, $\dist(\bm{z},\bm{x})\leq\norm{\bm{x}}/24$, the following
  \begin{equation*}
    \bm{w}^T\nabla^2 f(\bm{z})\bm{w}\geq 0
  \end{equation*}
holds uniformly with probability $1-e^{-\alpha n}$ ($\alpha$ depends
on $m$), where $\bm{w}\in\mathbb{R}^n$ such that $\norm{\bm{w}}=1$.
\end{theorem}

\begin{remark}
  Since the proof of local convexity is based on the
  Lemma~\ref{lem:7}, the result holds only for sampling vectors
  followed Gaussian model. It is an open problem whether it holds for
  CDP model or not.
\end{remark}
From the analysis above, we find that the randomness assumption is
useful for theoretical analysis. Theorems~\ref{thm:2} and~\ref{thm:3} state the convexity of the expectation of $f(\bm{z})$ and the
strong convexity of function $f(\bm{z})$ when $\bm{z}$ is around a
solution of the phase retrieval problem. These results imply the
numerical algorithms will perform well in the neighborhood of a
solution. And the local geometrical convergence rate is followed from
the convex optimization. 

\section{Numerical Simulations}
\label{sec:numer-simul}

This section introduces numerical simulations to illustrate the
empirical sampling complexity to ensure the convergence to a solution
of the optimization approach and study the effectiveness of line
search methods, especially the LBFGS method. Since gradient descent
with stepsize $\alpha_k$ from Theorem~\ref{thm:main} takes more
iterations than that with stepsize $\alpha_k$ from Wolfe
conditions. For comparison, we only consider the steepest gradient
descent algorithm.
\subsection{Relative Error and Termination Condition}
\label{sec:error-meas-term}

We denote the solution to the problem \eqref{eq:1} as
$\bm{x}$, and $\bm{z}$ is the returned solution by the line
search method, the relative error is defined
as 
\begin{equation}
  \label{eq:18}
  relerr = \min_{\lvert c\rvert=1}\frac{\norm{c\bm{x}-\bm{z}}_2}{\norm{\bm{x}}_2},
\end{equation}
where $c$ is to get rid of the effect of the constant phase shift of
phase problem. From 
\begin{equation*}
\langle c\bm{x}-\bm{z},c\bm{x}-\bm{z}\rangle=\lvert
  c\rvert^2\norm{\bm{x}}_2^2-2\re \langle
  c\bm{x},\bm{z}\rangle+\norm{\bm{z}}_2^2, 
\end{equation*}
the constant $c$ is given by
\begin{equation*}
  c=\frac{\langle \bm{x},\bm{z}\rangle}{\lvert \langle \bm{x},\bm{z}\rangle\rvert }.
\end{equation*}
We consider a vector to be successfully recovered if the relative
error is below $10^{-5}$.

In line search methods, for stop criteria, we take the three options:
the number of iterations is $600$, the tolerances for relative change
of objective function and relative change of variable 
during one iteration both are $10^{-12}$. The number of stored vector
pairs for LBFGS method $s=2$. 

\subsection{Recovery Rate}
\label{sec:recovery-rate}
We begin by examining the empirical recovery rate of the line search
method (LBFGS is used) for recovering random Gaussian signal
$\bm{x}\in\mathbb{R}^n$ or $\mathbb{C}^n$ under the Gaussian and CDP
models with octanary pattern. In the one dimensional simulations, we
consider signals with length $n=512,1014$. For two dimensional tests,
the test signals are in size $n=n_1\times n_2 = 128\times 128$. All test signals are drawn from the
Gaussian distribution, that is $\bm{x}\sim N(0,I)$ when
$\bm{x}\in\mathbb{R}^n$ and $\bm{x}\sim N(0,I)+iN(0,I)$ when
$\bm{x}\in\mathbb{C}^n$. 
The algorithm is  tested for $17$ values of $m=\delta n$,
where $\delta=2,2.5,\ldots,9.5,10$ for Gaussian model. In CDP model,
algorithm is tested for $9$ values of $L$ from $2$ to $10$. We
report the empirical probability of success in the two models in Figure~\ref{fig:2}. The
empirical probability of success is an average over 1000 trials, we
generate 10 different test signals and corresponding random sampling matrix $A$ in Gaussian model or
random masks in CDP model and begin the algorithm from 100 random
initials for a fixed random sampling matrix. 

No matter what the sampling vectors are, in Gaussian model or CDP model, for a signal with length $n=1024$, the optimization method can
successfully recover the signal when $m=7n$, see Figure
~\ref{fig:2a}. Note that $m$ is the order $n\log n$, which matches our
analysis. This empirical sampling complexity is greater than that of
WF with spectral initialization, WF can recover the signal when
$m=4n$~\cite{Candes2015}. It is not surprising for this fact, the need
for more sampling vectors is at the price of beginning from a random
initialization.

Assume the signal is real and the sampling vectors are in CDP
model, the algorithm recovers the real signals with sampling patterns
$L=5$. It recovers the real signals with sampling patterns $L=8$
without the priori constraint of real-valuedness. Figure~\ref{fig:2b}
illustrate this observation.
\begin{figure*}
  \begin{subfigure}[t]{.5\textwidth}
    \centering
    \includegraphics{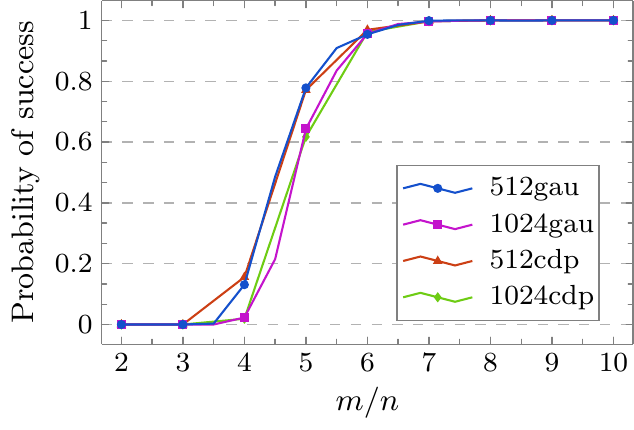}
    \caption{Recovery probability for 1-D signals in Gaussian and CDP models\label{fig:2a}}
  \end{subfigure}%
\hskip 1cm
 \begin{subfigure}[t]{.5\textwidth}
    \centering
    \includegraphics{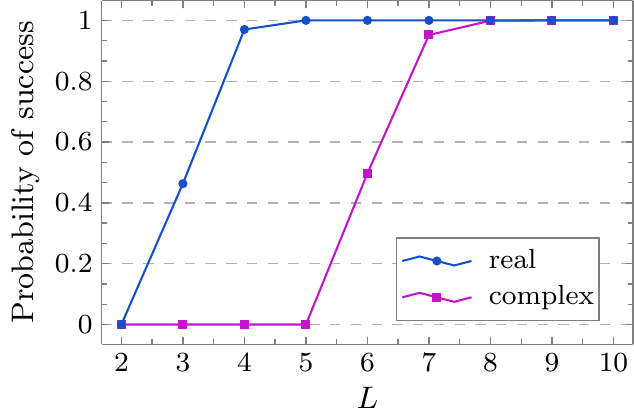}
    \caption{Recovery probability for 2-D signals in CDP model\label{fig:2b}}
  \end{subfigure}
  \caption{Average recovery probability\label{fig:2}}
\end{figure*}
\subsection{Alternative Projection}
\label{sec:altern-proj}

For comparison, we also apply the alternative projection algorithm to
phase retrieval problem. In the setting of multiple illuminations, the
algorithm is described in Algorithm~\ref{alg:2}.
\begin{algorithm}[H]
  \caption{Alternative Projection Algorithm\label{alg:2}}
  \begin{algorithmic}
    \REQUIRE initialization $\bm{z}_0$, maximum iterations $N$ and error $\epsilon$.
    \FOR {$k=0,\ldots,N$}
    \STATE update $\bm{c}$:
    \begin{equation}
      \label{eq:19}
      (\bm{c}_k)_i =
      \begin{cases}
        \frac{(A\bm{z}_k)_i}{\lvert (A\bm{z}_k)_i\rvert},& (A\bm{z}_k)_i\neq 0\\
        1,& (A\bm{z}_k)_i= 0
      \end{cases}
    \end{equation}
    \STATE update $\bm{z}$:
    \begin{equation}
      \label{eq:20}
      \bm{z}_{k+1} = \min \quad \norm{A\bm{z}-\bm{c}_k\sqrt{\bm{y}}}_2，
    \end{equation}
    \ENDFOR
  \end{algorithmic}
\end{algorithm}

where solution $\bm{z}_{k+1}$ to~\eqref{eq:20} is 
$A^{\dagger}\left(\bm{\bm{c}_k\sqrt{\bm{y}}}\right)$, where
\begin{equation}
  \label{eq:21}
  A^{\dagger}
  \begin{pmatrix}
    \bm{\omega}_1\\ \vdots\\\bm{\omega}_L
  \end{pmatrix}=\frac{1}{\sum_l \lvert \bm{d}_l\rvert^2}\left(\sum_{l=1}^L\bar{\bm{d}_l}\cdot F^*(\bm{\omega}_l)\right).
\end{equation}

The alternative projection method converges geometrically to a
solution of problem~\eqref{eq:1} under the Gaussian model with large
enough measurements, this fact is proven in
literature~\cite{Netrapalli2013}. If we apply it from a random
initialization, its overall performance is worse than optimization
approach. Figure~\ref{fig:6} illustrates this phenomenon. 
\subsection{Performance for the CDP Model}
\label{sec:perf-imag-cdp}
The test images is a complex-valued image of size $512\times 512$,
whose pixel values correspond to the complex transmission coefficients
of a collection of gold balls embedded in a medium. Its magnitude is
shown in Figure~\ref{fig:3a}. We only consider the CDP model, the stylized setup of
coded diffraction pattern, which one encounters in X-ray
crystallography and many other imaging sciences.
\subsubsection{Noise-free and Noise Measurements}
\label{sec:noise-free-meas}
 In the first experiment, we demonstrate the recovery of the image
 from noiseless measurements. We consider two different types of
 illuminations. The first type uses ten octanary masks. The reconstruction is shown in Figure~\ref{fig:3b}, It is visually
 indistinguishable from the original image. Since they are both
 complex-valued, we display only the magnitude. We also achieve
 successful recovery with eight octanary masks.

 Octanary masks may not be realizable in practice. Our second example
 uses simple random binary masks, where the entries are either $0$ or
 $1$ with equal probability. In this case, a large number of
 illuminations are required to achieve a reconstruction of comparable
 quality. The result for ten binary illuminations, one being regular
 Fourier measurement, is shown in Figure~\ref{fig:3c}.

  \begin{figure*}
   \begin{subfigure}[b]{.5\textwidth}
     \centering
\includegraphics{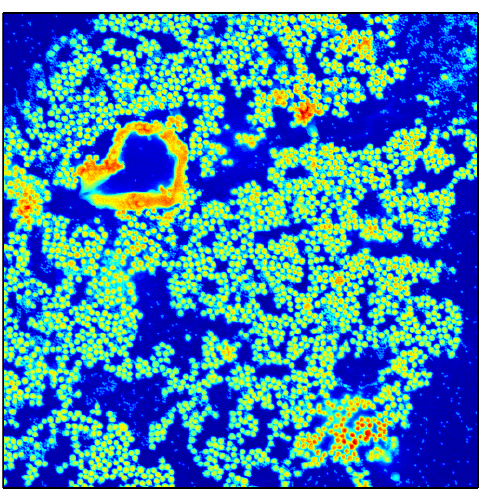}
     \caption{Original image\label{fig:3a}}
   \end{subfigure}
   \begin{subfigure}[b]{.5\textwidth}
\centering
\includegraphics{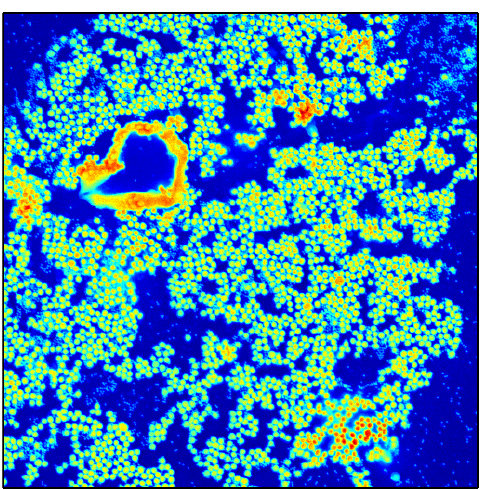}
     \caption{Reconstruction from ten octanary masks\label{fig:3b}}
   \end{subfigure}
  \begin{subfigure}[b]{.5\textwidth}
     \centering
\includegraphics{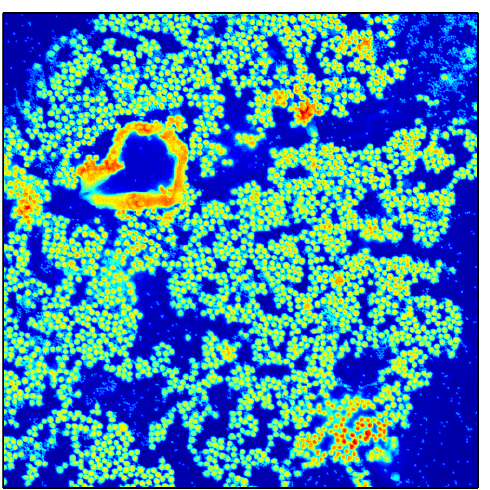}
     \caption{Reconstruction from ten binary masks\label{fig:3c}}
   \end{subfigure}
   \begin{subfigure}[b]{.5\textwidth}
\centering
\includegraphics{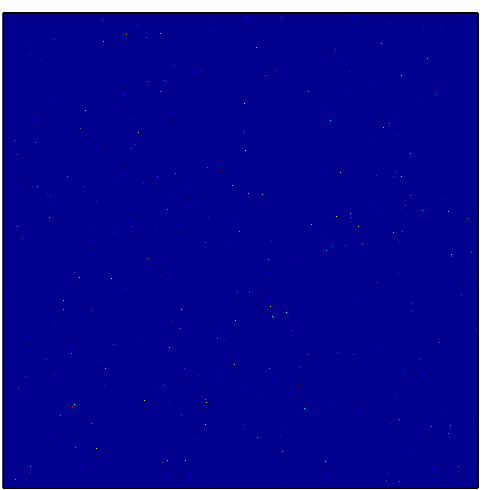}
     \caption{Error (a) and (c)\label{fig:3d}}
   \end{subfigure}
   \caption{Original gold balls image and reconstructions\label{fig:3}}
 \end{figure*}

In the second set of experiments we consider the same test image but
with noisy measurements. Ten octanary masks as before are used. Since
the main noise yields Poisson distribution resulting from the photon
counting in practice, we add
random Poisson noise to the measurements for ten different SNR levels,
ranging from 10dB to 55dB. Figure~\ref{fig:4} shows the
reconstructions from two SNR level data. Figure~\ref{fig:4a} and
~\ref{fig:4b} depict
the resulting reconstructions for low SNR case ($10$dB) and high SNR
case ($30$dB), respectively. Figure~\ref{fig:5} shows the average
relative error in dB versus the SNR\@. The error curves shows clearly
the linear behavior between SNR and relative error, it implies the
stability of the generalized phase retrieval problem~\cite{Candes2012}.
 \begin{figure*}
   \begin{subfigure}[b]{.5\textwidth}
     \centering
\includegraphics{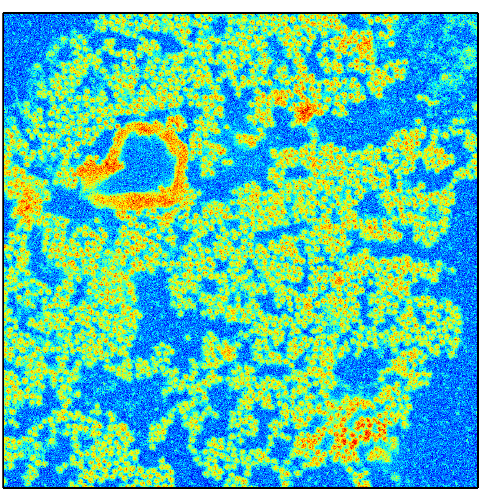}
     \caption{Low SNR =10 dB, relerr = 0.2098\label{fig:4a}}
   \end{subfigure}
   \begin{subfigure}[b]{.5\textwidth}
\centering
\includegraphics{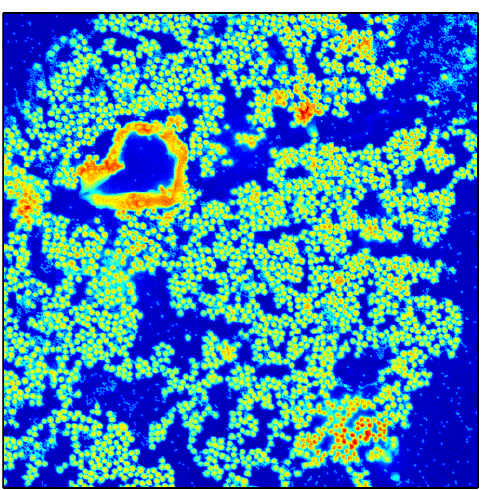}
     \caption{Low SNR =30 dB, relerr = 0.0209\label{fig:4b}}
   \end{subfigure}
   \caption{Reconstructions from noisy data\label{fig:4}}
 \end{figure*}

\begin{figure}
  \centering
  \includegraphics{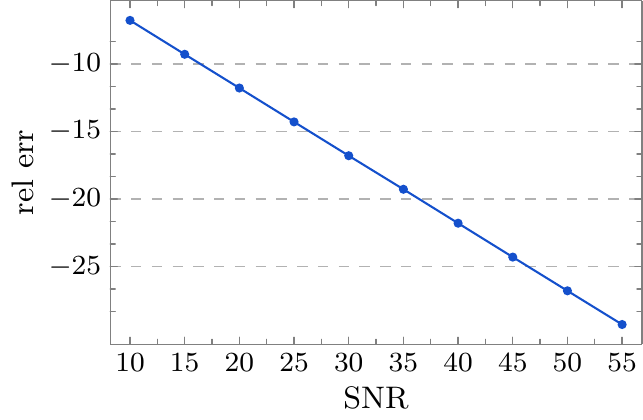}
\caption{Relative error in dB vs SNR\label{fig:5}}
\end{figure}
\subsubsection{Performance Comparison of Different Algorithms}
\label{sec:comp-line-search-1}

We study the performance of different line search method (SD, NCG and
LBFGS) and alternative projection method (AP). It shows the effectiveness
of the LBFGS to apply to the optimization problem. 
\begin{figure*}
   \begin{subfigure}[b]{.5\textwidth}
     \centering
     \includegraphics{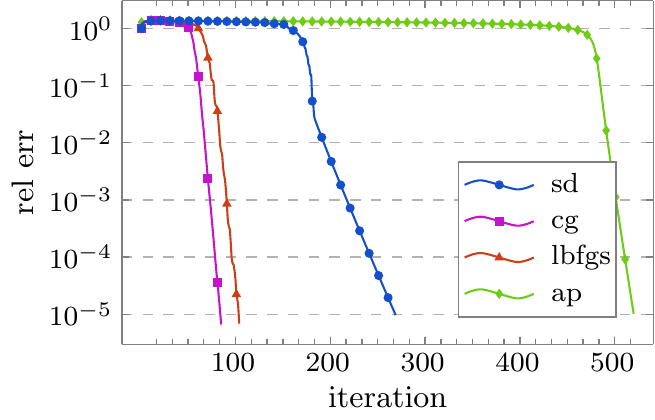}
     \caption{Performance versus different methods\label{fig:6a}}
   \end{subfigure}
   \begin{subfigure}[b]{.5\textwidth}
     \centering
     \includegraphics{fig6.pdf}
     \caption{Performance versus number of illuminations\label{fig:6b}}
   \end{subfigure}
   \caption{Comparison of different methods\label{fig:6}}
 \end{figure*}

Figure~\ref{fig:6a} depicts the relation between the relative error and
different methods. When the relative error is below $1e-5$, then the
algorithms break. The alternative projection (AP) decrease slowly at
the beginning $500$ iterations, when it is near the neighborhood of the
solution, it decreases fast. For line search methods, the nonlinear
conjugate gradient (NCG) and LBFGS have the competitive performance. In
terms of the number of iterations, NCG takes fewer iterations than
LBFGS\@. The overall calls of FFT are listed in Table~\ref{tab:1}. The
LBFGS algorithm need the fewest computations. So the LBFGS performs
best for the problem. This observation is the same as the authors'
previous paper. It is
shown that optimization approach is more robust than alternative
projection, where in 10 runs, AP fails for one run (the average is
excluded this run).

\begin{table}[H]
\caption{Total average number of FFT calls for different methods in 10 dependent runs}
\begin{center}
\begin{tabular}{ccccc}
\toprule
       Method  &      SD  &     NCG  &   LBFGS  &  AP  \\
\midrule
    FFT calls  &    24906  &     6438  &     3963  &    17527  \\
\bottomrule
\end{tabular}
\label{tab:1}
\end{center}
\end{table}

\subsubsection{Performance for Real-valued Signal}
\label{sec:real-assumptions}

With the aid of the constraint of real-valuedness, the phase retrieval
problem needs fewer diffraction patterns. Figure~\ref{fig:7} shows the
original test image and reconstructions from four octanary
illuminations. Table~\ref{tab:2} list the iterations that algorithm
returns the solution when given different number of masks. With the
real assumption, the more masks are provided, the fewer iterations
are needed. Even three masks can recover the solution. If without the
real assumption, six illuminations are needed and the number of iterations is over 1300.
 \begin{figure*}
   \begin{subfigure}[b]{.5\textwidth}
     \centering
\includegraphics{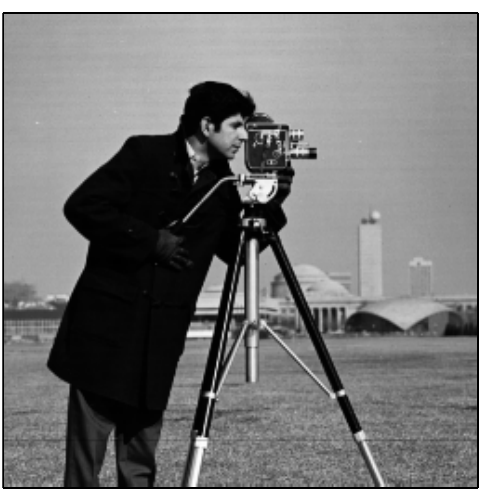}
     \caption{Original}
   \end{subfigure}
   \begin{subfigure}[b]{.5\textwidth}
\centering
\includegraphics{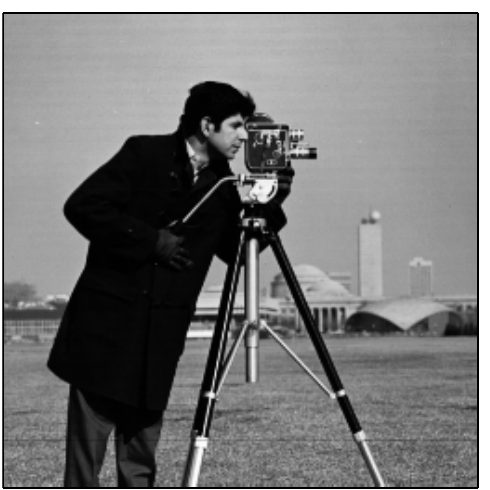}
     \caption{Reconstruction from four masks}
   \end{subfigure}
   \caption{Reconstructions of real-valued image\label{fig:7}}
 \end{figure*}

\begin{table}[H]
\caption{Cameraman ($256\times 256$) image recovery with/without the
  real-valued information of the signal using different number of
  masks: the array cell is the iteration number and - stands for
  failure of recovery within 3000 iterations.\label{tab:2}}
\begin{center}
\begin{tabular}{ccccc}
\toprule
       No.\ of masks &      3  &     4 &   5  &  6  \\
\midrule
    real &    166 &     94  &     71  &    62 \\
complex &    -  &     -  &     -  &    1319  \\
\bottomrule
\end{tabular}
\end{center}
\end{table}


\section{Conclusion and Discussion}
\label{sec:conclusion}
The generalized phase retrieval is solved by minimizing a nonconvex
least squares cost function with the advantage of better practice for two
dimensional problem. We deduced the expressions of gradient and
Hessian of the cost function by Fr\'{e}chet derivative
instead of Wirtinger derivative used in
literature~\cite{Candes2015}. Our approach is more convenient and
clear and can be applied to other complicated cost functions. Then we design
line search methods only depending gradient information for the real
functions with complex variables. The advantage is
that we do not need to carry out optimization algorithm with respect
to real and imaginary parts of variables. We proved the global
convergence of the proposed gradient decent algorithms, with arbitrary
random instead of special spectral initialization, under the
assumption that the sample vectors are drawn from complex Gaussian and
the sampling complexity $m$ is large enough ($\mathcal{O}(n\log n)$ for
Gaussian model and $\mathcal{O}(n\log^3 n)$ for CDP model).

 As demonstrated in this paper, a priori knowledge of the
real-valuedness of the
signal can mitigate the phase retrieval problem. For real-valued signal, the local convexity of the least
squares cost function in the neighborhood of the solution is also
proved. It is may be a hinder to obtain the appropriate stepsize in
the gradient descent method. For practical implementation and more
rapidly convergence rate, we apply other line search methods. The study of the 
performance of different line search methods shows that LBFGS
algorithm can efficiently and robustly solve the generalized phase retrieval
problem. The nonconvexity of the
cost function is not so scary is due to the premise that all local
minimizers associated with the cost function are global minimizers with
high probability.

Our theoretical results about global convergence are not applicable to the NCG
and LBFGS algorithms, while they show more appealing performance. The convergence of NCG and BFGS for general non-convex problems is not guaranteed~\cite{dai2002convergence}, and we cannot prove that for the generalized phase retrieval problem at present. Besides, the sampling
complexity may be further reduced to $\mathcal{O}(n)$ by borrowing the
idea of truncated Wirtinger flow in literate~\cite{Chen2015b}.
It is also 
worth investigating whether the performance of numerical algorithm
can be further improved by exploring the structures of the signal,
such as sparsity~\cite{Li2013}. Our theoretical analysis is based on
the random sampling vectors, unfortunately, the random measurements are not
easy to implement, so how to reduce the randomness is the main aspect
to study in phase retrieval community.
\section*{Acknowledgments}
The authors are indebted to Stefano Marchesini for providing us with the gold balls data set used in numerical
simulations. The first author would like to thank Ms. Chao Wang for helpful comments and suggestions on this manuscript draft. This work was supported in part by NSF grants of China (61421062, 11471024).
\appendix
 \section*{}
 \label{sec:appendix}
 \setcounter{section}{1} 

\subsection{Lemmas and Proofs}
\label{sec:lemmas}
\begin{lem}
\label{lem:2}
  For matrices $A\in\mathbb{C}^{m\times n}, C\in\mathbb{C}^{m\times
    n}$, and $B\in\mathbb{C}^{n\times m}, D\in\mathbb{C}^{n\times m}$,
  where $m\geq 2n$.
  Then we have
  \begin{multline*}
    \det\left(\lambda I-AB-CD\right)\\=\lambda^{m-2n}\det\left(
      \begin{bmatrix}
        \lambda I-BA&BC\\
        DA&\lambda I-DC
      \end{bmatrix}
\right).
  \end{multline*}
\end{lem}
\begin{proof}
  Consider matrices 
  \begin{equation*}
    \widetilde{A}=
    \begin{bmatrix}
      A&C
    \end{bmatrix}\in\mathbb{C}^{m\times 2n}\qquad \widetilde{B}=
    \begin{bmatrix}
      B\\D
    \end{bmatrix}\in\mathbb{C}^{2n\times m}
  \end{equation*}
According to the equality $\det(\lambda I-\widetilde{A}
\widetilde{B})=\lambda^{m-2n}\det(\lambda I-\widetilde{B}
\widetilde{A})$, we have
\begin{multline*}
    \det(\lambda I-\widetilde{A}
\widetilde{B})=\det\left(\lambda I-AB-CD\right)\\=\lambda^{m-2n}\det\left(
      \begin{bmatrix}
        \lambda I-BA&-BC\\
        -DA&\lambda I-DC
      \end{bmatrix}
\right)\\=\lambda^{m-2n}\det(\lambda I-\widetilde{B}
\widetilde{A}).
  \end{multline*}
\end{proof}

\begin{proof}[Proof of Lemma~\ref{lem:3}]
  we denote $\mathbb{E}[\nabla^2 f(\bm{x})]$ by $A$, the characteristic polynomial of the matrix $A$ is
  \begin{multline*}
    P(\lambda)=\det\left(\lambda I-A\right)\\=\det\left(\left(\lambda-\norm{x}^2\right)I-\frac{3}{2}\bm{a}\bm{a}^*-\frac{1}{2}\bm{b}\bm{b}^*\right),
  \end{multline*}
where $\bm{a}=
\begin{bmatrix}
  \bm{x}\\\bar{\bm{x}}
\end{bmatrix}\in\mathbb{C}^{2n}
$ and $\bm{b}=i
\begin{bmatrix}
  \bm{x}\\-\bar{\bm{x}}
\end{bmatrix}\in\mathbb{C}^{2n}
$.

By Lemma~\ref{lem:2}, we have
\begin{align*}
  P(\lambda)&=\det\left((\lambda-\norm{x}^2)I-\frac{3}{2}\bm{a}\bm{a}^*-\frac{1}{2}\bm{b}\bm{b}^*\right)\\
  &\hspace*{-1cm}=(\lambda-\norm{\bm{x}}^2)^{2n-2}\det\left((\lambda-\norm{\bm{x}}^2)I-\begin{bmatrix}
      3\norm{\bm{x}}^2&0\\
      0&-\norm{\bm{x}}^2
    \end{bmatrix}
\right)\\
  &\hspace*{-1cm}=(\lambda-\norm{\bm{x}}^2)^{2n-2}\lambda(\lambda-4\norm{\bm{x}}^2).
\end{align*}
So the eigenvalues are directly obtained. The proof for real case is similar.
\end{proof}

\begin{proof}[Proof of Lemma~\ref{lem:8}]
  According the expression of gradient
  \begin{equation*}
     \nabla f(\bm{z}) =\frac{1}{m}\sum_{r=1}^m \left(\lvert
                       \bm{a}_r^*\bm{z}\rvert^2-\lvert \bm{a}_r^*\bm{x}\rvert^2\right)\bm{a}_r\bm{a}_r^*\bm{z},
  \end{equation*}
and the expectation
\begin{equation*}
  \mathbb{E}[\lvert\bm{a}_r^*\bm{z}\rvert^2\bm{a}_r\bm{a}_r^*]=\norm{\bm{z}}^2I+\bm{z}\bm{z}^*,\quad
  \forall \bm{z}\in\mathbb{C}^n.
\end{equation*}
So
\begin{align*}
  \mathbb{E}[\nabla
  f(\bm{z})]&=(\norm{\bm{z}}^2I+\bm{z}\bm{z}^*)\bm{z}-(\norm{\bm{x}}^2I+\bm{x}\bm{x}^*)\bm{z}\\
&=(2\norm{\bm{z}}^2-\norm{\bm{x}}^2)\bm{z}-(\bm{x^*}\bm{z})\bm{x}.
\end{align*}
This completes the proof.
\end{proof}

\begin{proof}[Proof of Lemma~\ref{lem:9}]
  Let $\bm{z}=\bm{x}+t\bm{h}$ such that
  $\re(\bm{x}^*\bm{h})=\norm{\bm{x}},\norm{\bm{h}}=1$ and $t\geq 0$. Then
  \begin{equation*}
    \norm{\nabla f(\bm{z})-\mathbb{E}[\nabla
      f(\bm{z})]}=\max_{\bm{u}\in\mathbb{C}^n,\norm{\bm{u}}=1}\left\lvert\left\langle\bm{u}, \nabla f(\bm{z})-\mathbb{E}[\nabla
      f(\bm{z})]\right\rangle\right\rvert.
  \end{equation*}
By calculus, we have
\begin{align}
  \left\langle \bm{u}, \nabla
  f(\bm{z})\right\rangle\hspace*{-1cm}&\\
                                      &=\frac{1}{m}\sum_{r=1}^m
                          t\bm{u}^*\left(\lvert
                          \bm{a}_r^*\bm{x}\rvert^2\bm{a}_r\bm{a}_r^*\right)\bm{h}+t\bm{u}^*\left((\bm{a}_r^*\bm{x})^2\bm{a}_r\bm{a}_r^T\right)\overline{\bm{h}}\nonumber\\
&+2t^2\bm{u}^*\left(\lvert
                          \bm{a}_r^*\bm{h}\rvert^2\bm{a}_r\bm{a}_r^*\right)\bm{x}+t^2\bm{u}^*\left((\bm{a}_r^*\bm{h})^2\bm{a}_r\bm{a}_r^T\right)\overline{\bm{x}}\\
&+t^3\bm{u}^*\left(\lvert
                          \bm{a}_r^*\bm{h}\rvert^2\bm{a}_r\bm{a}_r^*\right)\bm{h}.\label{eq:770}
\end{align}
By Lemma~\ref{lem:8}, we have
\begin{multline*}
  \label{eq:772}
  \mathbb{E}[\nabla
    f(\bm{z})]=2t(t^2+3t\norm{\bm{x}}+2\norm{\bm{x}}^2)\bm{h}, \text{
      and }\\\norm{\mathbb{E}[\nabla
    f(\bm{z})]}=2t(t^2+3t\norm{\bm{x}}+2\norm{\bm{x}}^2).
\end{multline*}
Furthermore, we also have
\begin{align}
  \left\langle\bm{u},\mathbb{E}[\nabla f(\bm{z})]
  \right\rangle\hspace*{-1cm}&\\
                                      &=t\bm{u}^*\left(\norm{\bm{x}}^2I+\bm{x}\bm{x}^*\right)\bm{h}+t\bm{u}^*\left(2\bm{x}\bm{x}^T\right)\overline{\bm{h}}\nonumber\\
&+2t^2\bm{u}^*\left(\norm{\bm{h}}^2I+\bm{h}\bm{h}^*\right)\bm{x}+t^2\bm{u}^*\left(2\bm{h}\bm{h}^T\right)\overline{\bm{x}}\\&+t^3\bm{u}^*\left(\norm{\bm{h}}^2I+\bm{h}\bm{h}^*\right)\bm{h},\label{eq:771}
\end{align}
where we use the equality
$2\norm{\bm{h}}^2\bm{u}^*\bm{h}=\bm{u}^*\left(\norm{\bm{h}}^2I+\bm{h}\bm{h}^*\right)\bm{h}$.

Combining the two equalities \eqref{eq:770} and \eqref{eq:771} together and using triangular inequality and Lemma~\ref{lem:4} give
\begin{align*}
  \left\lvert\Bigl\langle\bm{u}, \nabla f(\bm{z})-\mathbb{E}[\nabla
      f(\bm{z})]\Bigr\rangle\right\rvert\hspace*{-3cm}&\\
&\leq
                                          t\left\lvert\bm{u}^*\Bigl
                                          (\frac{1}{m}\sum_{r=1}^m\lvert\bm{a}_r^*\bm{x}\rvert^2\bm{a}_r\bm{a}_r^*-\left(\norm{\bm{x}}^2I+\bm{x}\bm{x}^*\right)\Bigr
                                          )\bm{h}\right\rvert\\
&\phantom{==}{}+t\left\lvert\bm{u}^*\Bigl
                                          (\frac{1}{m}\sum_{r=1}^m(\bm{a}_r^*\bm{x})^2\bm{a}_r\bm{a}_r^T-2\bm{x}\bm{x}^T\Bigr
                                          )\overline{\bm{h}}\right\rvert\\
&\phantom{==}{}+2t^2\left\lvert\bm{u}^*\Bigl
                                          (\frac{1}{m}\sum_{r=1}^m\lvert\bm{a}_r^*\bm{h}\rvert^2\bm{a}_r\bm{a}_r^*-\left(\norm{\bm{h}}^2I+\bm{h}\bm{h}^*\right)\Bigr
                                          )\bm{x}\right\rvert\\
&\phantom{==}{}+t^2\left\lvert\bm{u}^*\Bigl
                                          (\frac{1}{m}\sum_{r=1}^m(\bm{a}_r^*\bm{h})^2\bm{a}_r\bm{a}_r^T-2\bm{h}\bm{h}^T\Bigr
                                          )\overline{\bm{x}}\right\rvert\\
&\phantom{==}{}+t^3\left\lvert\bm{u}^*\Bigl
                                          (\frac{1}{m}\sum_{r=1}^m\lvert\bm{a}_r^*\bm{h}\rvert^2\bm{a}_r\bm{a}_r^*-\left(\norm{\bm{h}}^2I+\bm{h}\bm{h}^*\right)\Bigr
                                          )\bm{h}\right\rvert\\
                                          &\leq
                                          2t\delta\norm{\bm{x}}^2+3t^2\delta\norm{\bm{x}}+t^3\delta\\
&=\frac{\delta}{2}\norm{\mathbb{E}[\nabla
      f(\bm{z})]}.
\end{align*}

Take $t=\dist(\bm{z},\bm{x})\leq \norm{\bm{x}}/2$,
then
\begin{equation*}
  t^2+3t\norm{\bm{x}}+2\norm{\bm{x}}^2\leq
  \frac{15}{4}\norm{\bm{x}}^2\leq 4\norm{\bm{x}}^2.
\end{equation*}
So
\begin{equation*}
  \left\lvert\Bigl\langle\bm{u}, \nabla f(\bm{z})-\mathbb{E}[\nabla
      f(\bm{z})]\Bigr\rangle\right\rvert \leq 4\delta\dist(\bm{z},\bm{x})\norm{\bm{x}}^2.
\end{equation*}
\end{proof}

\begin{proof}[Proof of Lemma~\ref{lem:10}]
Since we have
\begin{equation*}
  \mathbb{E}[\nabla
    f(\bm{z})]=2t(t^2+3t\norm{\bm{x}}+2\norm{\bm{x}}^2)\bm{h},
\end{equation*}
it is obvious that $\mathbb{E}[\nabla
 f(\bm{z})]$ and $\bm{h}$ share the same direction. Furthermore,
we have
\begin{multline*}
    \norm{\mathbb{E}[\nabla
    f(\bm{z})]}=2t(t^2+3t\norm{\bm{x}}+2\norm{\bm{x}}^2)\\=2t\left((t+\frac{3}{2}\norm{\bm{x}})^2-\frac{1}{4}\norm{\bm{x}}^2\right)\geq
    0,
  \end{multline*}
so  $\mathbb{E}[\nabla f(\bm{z})]=\bm{0}$ if and only if $\bm{z}$ is a
 solution of the phase retrieval problem \eqref{eq:1}. Furthermore,
by equation~\eqref{eq:23}, it follows that
\begin{equation*}
  \re\left\langle \nabla
      f(\bm{z}),\bm{z}-\bm{x}e^{i\phi(\bm{z})}\right\rangle\geq \sqrt{1-\frac{\delta^2}{4}}
    \norm{\nabla f(\bm{z})}\norm{\bm{z}-\bm{x}e^{i\phi(\bm{z})}},
\end{equation*}
so the angle between $\nabla f(\bm{z})$ and
  $\bm{z}-\bm{x}e^{i\phi(\bm{z})}$ is below $\arcsin(\delta/2)$.
\end{proof}

\subsection{Proof of Theorem~\ref{thm:main}}
\label{sec:proof-theorem}
\begin{proof}
According to Lemma~\ref{lem:10}, the angle between
$\bm{z}_k-\bm{x}e^{i\phi(\bm{z}_k)}$ and $\nabla f(\bm{z}_k)$ is below
$\pi/2$. It is obvious that if $0<\alpha_k<2\re\langle\bm{z}_k-\bm{x}e^{i\phi(\bm{z}_k)},\nabla f(\bm{z}_k)\rangle/\norm{\nabla
      f(\bm{z}_k)}$, according to the definition of
    $\phi(\bm{z}_{k+1})$, we have
    \begin{equation*}
      \norm{\bm{z}_{k+1}-\bm{x}e^{i\phi(\bm{z}_{k+1})}}\leq \norm{\bm{z}_{k+1}-\bm{x}e^{i\phi(\bm{z}_k)}}<\norm{\bm{z}_{k}-\bm{x}e^{i\phi(\bm{z}_k)}},
    \end{equation*}
i.e., the strict descent property $\dist(\bm{z}_{k+1},\bm{x})< \dist(\bm{z}_{k},\bm{x})$ holds.

So we have to bound the
$\norm{\bm{z}_{k}-\bm{x}e^{i\phi(\bm{z}_k)}}$. For brevity, we omit the
subscript $k$ without ambiguity and assume the $\phi(\bm{z})=0$.  Let
$\bm{z}=\bm{x}+t\bm{h}$ such that
$\re(\bm{x}^*\bm{h})=\norm{\bm{x}},\norm{\bm{h}}=1$ and $t\geq 0$, so
$\dist(\bm{z},\bm{x})=t$. According to
  \eqref{eq:23} in Lemma~\ref{lem:9}, we have
  \begin{equation*}
    \norm{\nabla f(\bm{z})}\leq
    (1+\frac{\delta}{2})\norm{\mathbb{E}[\nabla f(\bm{z})]}\leq (2+\delta)t(t^2+3t\norm{\bm{x}}+2\norm{\bm{x}}^2).
  \end{equation*}
We have
\begin{equation*}
  \norm{\nabla f(\bm{z})}\leq
  \begin{cases}
    6(2+\delta)\norm{\bm{x}}^2t,& t\leq \norm{\bm{x}};\\
    6(2+\delta)t^3,& t>\norm{\bm{x}}.
  \end{cases}
\end{equation*}
Consequently, we have
\begin{equation*}
  t\geq
  \begin{cases}
    \frac{\norm{\nabla f(\bm{z})}}{6(2+\delta)\norm{\bm{x}}^2},& t\leq \norm{\bm{x}};\\
    \sqrt[3]{\frac{\norm{\nabla f(\bm{z})}}{6(2+\delta)}},& t>\norm{\bm{x}}.
  \end{cases}
\end{equation*}
Thus, according to Lemma~\ref{lem:10}, we have
\begin{multline*}
  2\re\langle\bm{z}-\bm{x}e^{i\phi(\bm{z})},\nabla f(\bm{z})\rangle/\norm{\nabla
      f(\bm{z})}\\\geq 2\sqrt{1-\frac{\delta^2}{4}}\norm{\bm{z}-\bm{x}e^{i\phi(\bm{z})}}.
\end{multline*}
It follows that
\begin{multline*}
  0<\alpha_k\\<\operatorname{min}\left\{2\sqrt{1-\frac{\delta^2}{4}}\frac{\norm{\nabla f(\bm{z}_k)}}{6(2+\delta)\norm{\bm{x}}^2},2\sqrt{1-\frac{\delta^2}{4}}\sqrt[3]{\frac{\norm{\nabla f(\bm{z}_k)}}{6(2+\delta)}}\right\}.
\end{multline*}

For the proof of convergence, the stepsize $\alpha_k$ can not be too small. We denote the intersection of the polynomial of
degree three $P_2(t)=(2-\delta)t(t^2+3t\norm{\bm{x}}+2\norm{\bm{x}}^2)$ and
$y=\norm{\nabla f(\bm{z}_k)}$ as $t_2$. We choose
$\alpha_k=\sqrt{1-\delta^2/4}t_1$. The proof of convergence is
followed. According to
  \eqref{eq:23} in Lemma~\ref{lem:9}, since we have the following inequality
\begin{equation}
\label{eq:1000}
  (1-\frac{\delta}{2})\norm{\mathbb{E}[\nabla f(\bm{z}_k)]}\leq
  \norm{\nabla f(\bm{z}_k)}\leq (1+\frac{\delta}{2})\norm{\mathbb{E}[\nabla f(\bm{z}_k)]},
\end{equation}
where the quality
\begin{equation*}
  \norm{\mathbb{E}[\nabla f(\bm{z}_k)]}=2t(t^2+3t\norm{\bm{x}}+2\norm{x}^2).
\end{equation*}
It is obvious that from \eqref{eq:1000} we can bound the variable $t$
for a given $\norm{\nabla f(\bm{z}_k)}$.

From Figure \ref{fig:1}, we have that
\begin{equation*}
  \frac{\lvert ST\rvert}{\lvert OT\rvert}\leq \frac{\lvert ST\rvert}{\lvert OT_1\rvert}.
\end{equation*}
We take $S=(t_1\cos \beta,0)$ and $T=(t\cos \beta,t\sin \beta)$, where $t\in[t_1,t_2]$. It follows that
\begin{align}
  \frac{\lvert ST\rvert^2}{\lvert OT_1\rvert^2}&=\frac{(t-t_1)^2\cos^2\beta+t^2\sin^2\beta}{t_1^2}\nonumber\\
  &=\left(\frac{t-t_1}{t_1}\right)^2\cos^2\beta+\left(1+\frac{t-t_1}{t_1}\right)^2\sin^2\beta\nonumber\\
&\leq \left(\frac{t-t_1}{t_1}\right)^2\cos^2\beta+2\sin^2\beta+2\left(\frac{t-t_1}{t_1}\right)^2\sin^2\beta\nonumber\\
&\leq 2\left(\frac{t-t_1}{t_1}\right)^2+2\sin^2\beta.\label{eq:1001}
\end{align}
Since the $P_1$ and $P_2$ are both convex when $t\geq 0$, we have the inequality
\begin{equation*}
  \frac{t-t_1}{t_1}\leq \frac{2\delta (t_1^2+3t_1+2)}{(2-\delta)(3t_1^2+6t_1+2)}\colon = f(t_1),
\end{equation*}
where we assume $\norm{\bm{x}}=1$ without loss of generality. 
It is obvious that $f(t_1)$ monotonously decreases first and then
monotonously increases. As $f(t_1)\to 2\delta/(3(2-\delta))$,
$t_1\to\infty$, we have
\begin{multline*}
  \frac{\lvert ST\rvert^2}{\lvert OT_1\rvert^2}\leq 2\left(\frac{2\delta}{2-\delta}\right)^2+2\sin^2\beta\\\leq 2\left(\frac{2\delta}{2-\delta}\right)^2+\frac{\delta^2}{2}\leq \left(\frac{2\sqrt{2}\delta}{2-\delta}+\frac{\delta}{\sqrt{2}}\right)^2.
\end{multline*}
That is
\begin{equation*}
  \lvert
  ST\rvert\leq\left(\frac{2\sqrt{2}\delta}{2-\delta}+\frac{\delta}{\sqrt{2}}\right)
  \lvert OT_1\rvert\leq \left(\frac{2\sqrt{2}\delta}{2-\delta}+\frac{\delta}{\sqrt{2}}\right)
  \lvert OT\rvert,
\end{equation*}
i.e., the convergence is ensured. It yields that
\begin{equation*}
  \dist(\bm{z}_{k+1},\bm{x})\leq\left(\frac{2\sqrt{2}\delta}{2-\delta}+\frac{\delta}{\sqrt{2}}\right)\dist(\bm{z}_{k},\bm{x}).
\end{equation*}
Since the $g(\delta)=\frac{6\delta-\delta^2}{\sqrt{2}(2-\delta)}$
monotonously increases, we can find small enough $\delta$ to satisfy the condition $g(\delta)<1$, such as $\delta\leq 0.2$, then $g(\delta)<0.5$.
\end{proof}
\begin{figure}
  \centering
  \includegraphics[width=.45\textwidth]{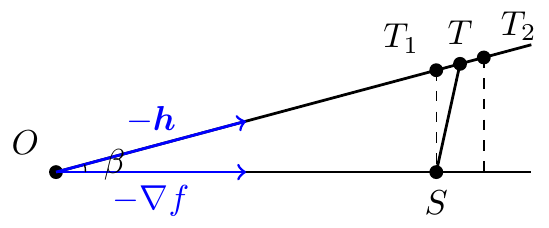}
  \caption{Illustration of the descent property: point $T$ is between
    $T_1$ and $T_2$, and $\lvert OT_1\rvert=t_1$ and $\lvert
    OT_2\rvert=t_2$, the step size $\lvert OS\rvert=t_1\sin \beta$,
    where $\beta$ is the angle between $-\bm{h}$ and $-\nabla f$.}
  \label{fig:1}
\end{figure}
\begin{proof}[Proof of Corollary~\ref{cor:cor}]
  If we have $t\leq \frac{1}{5}\norm{\bm{x}}$, then it follows that
\begin{equation*}
  \norm{\nabla f(\bm{z}_k)}\leq (2+\delta)t(t^2+3t\norm{\bm{x}}+2\norm{\bm{x}}^2)\leq (2+\delta)\frac{66}{25}\norm{\bm{x}}^2t.
\end{equation*}
Then we have
\begin{equation*}
  t\geq \frac{25\norm{\nabla f(\bm{z}_k)}}{66(2+\delta)\norm{\bm{x}}^2}.
\end{equation*}

We have the estimations
\begin{equation*}
  \frac{t-t_1}{t_1}\leq\frac{66(2+\delta)/25-(2-\delta)(t_1^2+3t_1+2)}{(2-\delta)(3t_1^2+6t_1+2)}\leq \frac{58\delta+16}{25(2-\delta)},
\end{equation*}
and
\begin{equation*}
  \frac{\lvert ST\rvert^2}{\lvert OT_1\rvert^2}\leq 2\left(\frac{58\delta+16}{25(2-\delta)}\right)^2+\frac{\delta^2}{2}\leq \left(\frac{\sqrt{2}(58\delta+16)}{25(2-\delta)}+\frac{\delta}{\sqrt{2}}\right)^2.
\end{equation*}
For the case $\dist(\bm{z}_k,\bm{x})\geq 2\norm{\bm{x}}$, the proof is
similar, we omit it.
\end{proof}

\subsection{Proof of Local Convexity}
\label{sec:proof-local-conv}
\begin{proof}[Proof of Theorem~\ref{thm:2}]
Taking an arbitrary $\bm{z}\in\mathbb{R}^n$, we let
  $\bm{z}=\bm{x}+t\bm{w}$ ($\norm{w}=1,t\geq 0$). According to
  Lemma~\ref{lem:5}, it yields
  \begin{multline*}
    \mathbb{E}[\nabla^2 f(\bm{z})]\\
    =\mathbb{E}[\nabla^2
                                    f(\bm{x}+t\bm{w})]
=\left(I+4\bm{w}\bm{w}^T\right)t^2\\+4\left((\bm{x}^T\bm{w})I+\bm{x}\bm{w}^T+\bm{w}\bm{x}^T\right)t+\left(\norm{\bm{x}}^2I+3\bm{x}\bm{x}^T\right).
  \end{multline*}
Observe that
\begin{equation*}
  (\bm{x}^T\bm{w})I+\bm{x}\bm{w}^T+\bm{w}\bm{x}^T\succeq -3\norm{\bm{x}}I
\end{equation*}
and
\begin{equation*}
  \norm{\bm{x}}^2I+3\bm{x}\bm{x}^T\succeq \norm{\bm{x}}^2I.
\end{equation*}
So, when $t\leq \norm{\bm{x}}/12$, the expectation
$\mathbb{E}[\nabla^2 f(\bm{z})]\succeq 0$. This completes the proof.
\end{proof}
\begin{lem}
\label{lem:6}
  For a given vector $\bm{z}\in\mathbb{R}^n$, we parameterize it by
  $\bm{z}=\bm{x}+t\bm{w}$, where $\bm{w}=1,t\geq 0$. Then according to
  the chain rule of Wirtinger derivatives, we have
  \begin{equation*}
    f(\bm{z})=f(\bm{x}+t\bm{w})=\frac{1}{2m}\norm{\lvert A(\bm{x}+t\bm{w})\rvert^2-\bm{y}}^2.
  \end{equation*}
Its gradient and Hessian with respect to $t$ are:
\begin{subequations}
  \begin{align*}
    f'(t)&\\
    &=\frac{1}{m}\sum_{r=1}^m\begin{bmatrix}
         \bm{w}^T&\bm{w}^*
       \end{bmatrix}
                   \begin{bmatrix}
         \left(\lvert
                       \bm{a}_r^*(\bm{x}+t\bm{w})\rvert^2-y_r\right)\overline{(\bm{a}_r^*(\bm{x}+t\bm{w}))}\overline{\bm{a}_r} \\       
 \left(\lvert
                       \bm{a}_r^*(\bm{x}+t\bm{w})\rvert^2-y_r\right)(\bm{a}_r^*(\bm{x}+t\bm{w}))\bm{a}_r
                   \end{bmatrix}\\
&=\frac{2}{m}\sum_{r=1}^m\re\Bigl(\left(\lvert
           \bm{a}_r^*(\bm{x}+t\bm{w})\rvert^2-y_r\right)\bm{a}_r^*(\bm{x}+t\bm{w})\bm{w}^*\bm{a}_r\Bigr),\\
f''(t)&=\frac{1}{m}\sum_{r=1}^m
       \begin{bmatrix}
         \bm{w}^*&\bm{w}^T
       \end{bmatrix}\\
                   &\hspace*{-1cm}\begin{bmatrix}
                     \left(2\lvert
                       \bm{a}_r^*(\bm{x}+t\bm{w})\rvert^2-y_r\right)\bm{a}_r\bm{a}_r^*&
                     \left(\bm{a}_r^*(\bm{x}+t\bm{w})\right)^2\bm{a}_r\bm{a}_r^T\\
                     \left(\overline{\bm{a}_r^*(\bm{x}+t\bm{w})}\right)^2\overline{\bm{a}_r}\bm{a}_r^*&\left(2\lvert
                       \bm{a}_r^*(\bm{x}+t\bm{w})\rvert^2-y_r\right)\overline{\bm{a}_r}\bm{a}_r^T
                   \end{bmatrix}                                                                   \begin{bmatrix}
                                  \bm{w}\\\overline{\bm{w}}         
               \end{bmatrix}\\
&=\frac{2}{m}\sum_{r=1}^m\re\Bigl(\left(2\lvert
  \bm{a}_r^*(\bm{x}+t\bm{w})\rvert^2-y_r\right)\lvert
  \bm{a}_r^*\bm{w}\rvert^2\\
&+\left(\bm{a}_r^*(\bm{x}+t\bm{w})\right)^2\left(\bm{a}_r^T\overline{\bm{w}}\right)^2\Bigr).
\end{align*}
\end{subequations}
\end{lem}
\begin{proof}[Proof of Theorem~\ref{thm:3}]
By Lemma~\ref{lem:6}, we have
  \begin{multline*}
  \frac{1}{2}
\begin{bmatrix}
       \bm{w}^*&\bm{w}^T
       \end{bmatrix}\nabla^2 f(\bm{z})\begin{bmatrix}
                                  \bm{w}\\\overline{\bm{w}}         
               \end{bmatrix}
    =\frac{1}{2}f''(t)\\=\frac{1}{m}\sum_{r=1}^m \re\Bigl(3\lvert \bm{a}_r^*\bm{w}\rvert^4t^2+4\re\left(\bm{w}^*\bm{a}_r\bm{a}_r^*\bm{x}\right)\lvert \bm{a}_r^*\bm{w}\rvert^2 t\Bigr.\\\Bigl.+2(\bm{w}^*\bm{a}_r\bm{a}_r^*\bm{x})\lvert \bm{a}_r^*\bm{w}\rvert^2 t\Bigr)+\frac{1}{2}f''(0)\\
=\frac{1}{m}\sum_{r=1}^m3\lvert \bm{a}_r^*\bm{w}\rvert^4t^2+6\re\left(\bm{w}^*\bm{a}_r\bm{a}_r^*\bm{x}\right)\lvert \bm{a}_r^*\bm{w}\rvert^2 t\\+ \lvert\bm{a}_r^*\bm{x}\rvert^2\lvert \bm{a}_r^*\bm{w}\rvert^2+\re\Bigl(\left(\bm{w}^*\bm{a}_r\bm{a}_r^*\bm{x}\right)^2\Bigr)\\
=\frac{1}{m}\sum_{r=1}^m3\lvert \bm{a}_r^*\bm{w}\rvert^4t^2+6\re\left(\bm{w}^*\bm{a}_r\bm{a}_r^*\bm{x}\right)\lvert \bm{a}_r^*\bm{w}\rvert^2 t\\+2\Bigl(\re\left(\bm{w}^*\bm{a}_r\bm{a}_r^*\bm{x}\right)\Bigr)^2,
  \end{multline*}
where we use the equality $2\re(c)^2=\lvert c\rvert^2+\re(c^2)$.

We can denote it 
\begin{multline*}
   \frac{1}{2}
\begin{bmatrix}
       \bm{w}^*&\bm{w}^T
       \end{bmatrix}\nabla^2 f(\bm{z})\begin{bmatrix}
                                  \bm{w}\\\overline{\bm{w}}         
               \end{bmatrix}\\
    =\frac{3}{m}\sum_{r=1}^m (A_rt+B_r)^2-
 \frac{1}{4}
\begin{bmatrix}
       \bm{w}^*&\bm{w}^T
       \end{bmatrix}\nabla^2 f(\bm{x})\begin{bmatrix}
                                  \bm{w}\\\overline{\bm{w}}         
               \end{bmatrix},
\end{multline*}
where
\begin{align*}
  A_r &= \lvert \bm{a}_r^*\bm{w}\rvert^2\\
  B_r &= \re\left(\bm{w}^*\bm{a}_r\bm{a}_r^*\bm{x}\right).
\end{align*}
Define
\begin{equation*}
  Z_r(t) = (A_rt+B_r)^2\geq 0,
\end{equation*}
according to Lemma~\ref{lem:7}, we have
\begin{align*}
  \mathbb{E}[Z_r(t)]&=\mathbb{E}\Bigl[\lvert \bm{a}_r^*\bm{w}\rvert^4t^2+2\re\left(\bm{w}^*\bm{a}_r\bm{a}_r^*\bm{x}\right)\lvert \bm{a}_r^*\bm{w}\rvert^2 t\\&+\Bigl(\re\left(\bm{w}^*\bm{a}_r\bm{a}_r^*\bm{x}\right)\Bigr)^2\Bigr]\\
&=2t^2+4\re(\bm{w}^*\bm{x})t+
 \frac{1}{4}
\begin{bmatrix}
       \bm{w}^*&\bm{w}^T
       \end{bmatrix}\mathbb{E}[\nabla^2 f(\bm{x})]\begin{bmatrix}                                 \bm{w}\\\overline{\bm{w}}         
               \end{bmatrix}.
\end{align*}
In the real case, we have $\bm{w}^*=\bm{w}^T$ and
\begin{align*}
  \bm{w}^T\nabla^2 f(\bm{z})\bm{w}&=\frac{1}{2}f''(t)\\
&=\frac{3}{m}\sum_{r=1}^mZ_r(t)-\frac{1}{2}\bm{w}^T\nabla^2 f(\bm{x})\bm{w}.
\end{align*}
Consequently, it follows that the expectation 
\begin{equation*}
  \mathbb{E}[Z_r(t)]=2t^2+4\bm{w}^T\bm{x}t+
 \frac{1}{2}
\bm{w}^T\mathbb{E}[\nabla^2 f(\bm{x})]\bm{w},
\end{equation*}
and that the variance
\begin{equation*}
  \mathbb{E}[(Z_r(t)-\mathbb{E}[Z_r(t)])^2]\leq C^2(t),
\end{equation*}
where $C(t)$ depends only on the $\bm{a}_r$.

Applying Lemma 5.4 in literature~\cite{White2015}
yields (taking $y=m\lambda/12$)
\begin{multline*}
  \mathbb{P}\Bigl[\mathbb{E}[Z_r(t)]-\frac{1}{m}\sum_{r=1}^mZ_r(t)\geq
  \frac{\lambda}{12}\Bigr]\\
\leq\operatorname{min}\Bigl\{\exp\left(-\frac{\lambda^2m}{144C^2(t)}\right),25\left(1-\Phi\left(\frac{\lambda\sqrt{m}}{12C(t)}\right)\right)\Bigr\}.
\end{multline*}
Using the well-known bound
\begin{equation*}
  1-\Phi\left(\frac{\lambda\sqrt{m}}{12C(t)}\right)<\frac{12C(t)}{\lambda\sqrt{2m\pi}}\exp\left(-\frac{\lambda^2m}{288C^2(t)}\right),
\end{equation*}
we find that if $m\geq 288\alpha\lambda^{-2}C^2(t)n$ then with
probability at least $1-e^{-\alpha n}$ we have
\begin{align*}
  \frac{1}{2}f''(t)&\geq 3\mathbb{E}[Z_r(t)]-\lambda/4-\frac{1}{2}\bm{w}^T\nabla^2 f(\bm{x})\bm{w}\\
&\geq 6t^2+12\bm{w}^T\bm{x}t+
\bm{w}^T\mathbb{E}[\nabla^2 f(\bm{x})]\bm{w}-\lambda/4\\&+\frac{1}{2}\bm{w}^T\mathbb{E}[\nabla^2 f(\bm{x})]\bm{w}-\frac{1}{2}\bm{w}^T\nabla^2 f(\bm{x})\bm{w}\\
&\geq 6t^2+12\bm{w}^T\bm{x}t+\lambda/2,
\end{align*}
where we use the concentration of the Hessian around its mean
from Lemma~\ref{lem:4}.

Since the smallest eigenvalue of
$\mathbb{E}[\nabla^2 f(\bm{x})]$ is $\norm{\bm{x}}^2$ (see
Lemma~\ref{lem:3}), taking $\lambda=\norm{\bm{x}}^2$, it follows that
\begin{equation*}
   \frac{1}{2}f''(t)\geq 6t^2+12\bm{w}^T\bm{x}t+\norm{\bm{x}}^2/2.
\end{equation*}
So when $0\leq t\leq \norm{\bm{x}}/24$, $f''(t)\geq 0$
holds with high probability.
\end{proof}
\bibliography{phaseretrieval}
\end{document}